\documentclass[11pt]{article}
\usepackage{graphicx}
\usepackage{mathptmx}      
\usepackage{amsfonts}
\usepackage{amsmath}
\usepackage{amsthm}
\usepackage{amscd}
\usepackage{lmodern}
\usepackage{enumitem}
\usepackage[latin2]{inputenc}
\usepackage{t1enc}
\usepackage[mathscr]{eucal}
\usepackage{indentfirst}
\usepackage{mathptmx}
\usepackage{graphicx}
\usepackage{xr}

\usepackage[thinlines]{easytable}
\usepackage{graphics}
\usepackage{pict2e}
\usepackage{amssymb}
\usepackage{epic}
\usepackage{verbatim}
\numberwithin{equation}{section}
\usepackage{epstopdf}
\usepackage{caption}
\usepackage[a4paper, total={6in, 8in}]{geometry}
\usepackage{mdframed}
\usepackage{listings}
\usepackage{pifont}
\makeatother
\usepackage{subcaption}
\newcommand{\F}{\mathcal{F}}

\newcommand{\PWidth}{\tx{PWidth}}
\newcommand{\RR}{\mathbb{R}}
\newcommand{\DSB}{\tx{DSB}}

\newcommand{\Sc}[2]{\langle #1, #2 \rangle}

\newcommand{\bak}{\bar{\alpha}_k}

\newcommand{\tx}{\textnormal}
\newcommand{\BB}{\bar{B}}

\newcommand{\n}[1]{\|#1\|}
\newcommand{\argmin}{\tx{argmin}}
\newcommand{\argmax}{\tx{argmax}}
\newcommand{\xx}{\bar{x}}
\newcommand{\dist}{\tx{dist}}
\newcommand{\Ab}{\mathcal{A}}
\newcommand{\s}[2]{\langle #1, #2 \rangle}
\newcommand{\OM}{\Omega}
\newcommand{\K}{K}
\newcommand{\conv}{\tx{conv}}
\newcommand{\sm}{\setminus}
\newcommand{\SB}{\tx{SB}}

\newcommand{\bqfs}{q_{gs}}

\newcommand{\PFWidth}{\textnormal{PFWidth}}
\newcommand{\ck}{\sqrt{2 \mu}}
\newcommand{\pfaces}{\text{pfaces}}

\newcommand{\vf}{\text{vf}}

\newcommand{\qfs}{\bar{q}_{gs}}

\newtheorem{theorem}{Theorem}[section]
\newtheorem{assumption}{Assumption}[section]
\newtheorem{lemma}{Lemma}[section]
\newtheorem{corollary}{Corollary}[section]

\newtheorem{proposition}{Proposition}[section]

\newtheorem{example}{Example}%
\newtheorem{remark}{Remark}%

	\title{Avoiding bad steps in Frank-Wolfe variants}
\author{ Francesco~Rinaldi\thanks{Dipartimento di Matematica ``Tullio Levi-Civita'', Universit\`a
		di Padova, Italy
		(\tt{rinaldi@math.unipd.it})}
	\and 
	Damiano~Zeffiro\thanks{Dipartimento di Matematica ``Tullio Levi-Civita'', Universit\`a
		di Padova, Italy
		(\tt{damiano.zeffiro@math.unipd.it})}
}

\begin{document}
\maketitle

	\abstract{	The study of Frank-Wolfe (FW) variants is often complicated by the presence of different kinds of "good" and "bad" steps. In this article, we aim to simplify the convergence analysis of specific variants by getting rid of such a distinction between steps, and to improve existing rates by ensuring a non-trivial bound at each iteration.

		In order to do this, we define the Short Step Chain (SSC) procedure, which skips gradient computations in consecutive short steps until proper conditions are satisfied. This algorithmic tool allows us to give a unified analysis and converge rates in the general smooth non convex setting, as well as a linear convergence rate under a Kurdyka-\L ojasiewicz (KL) property. While the KL setting has been widely studied for proximal gradient type methods, to our knowledge, it has never been analyzed before for the Frank-Wolfe variants considered in the paper.

		An angle condition, ensuring that the directions selected by the methods have the steepest slope possible up to a constant, is used to carry out our analysis.
		We prove that such a  condition is satisfied, when considering minimization problems over a polytope, by the away step Frank-Wolfe (AFW), the pairwise Frank-Wolfe (PFW), and the Frank-Wolfe method with in face directions (FDFW). \\}

	\maketitle
		\section{Introduction}
	The Frank-Wolfe method \cite{frank1956algorithm} and its variants (see, e.g., \cite{freund2017extended}, \cite{lacoste2015global} and references therein) provide a valid alternative to projected gradient approaches for the constrained optimization of a smooth objective $f: \RR^n \rightarrow \RR$, in settings where projecting on the feasible set may be unpractical. These methods have found many applications in sparse and structured optimization (see, e.g., \cite{berrada2018deep}, \cite{freund2017extended}, \cite{jaggi2013revisiting}, \cite{joulin2014efficient}, 
\cite{osokin2016minding}  and references therein).

In this paper, we aim to overcome an annoying issue affecting the analysis of some FW variants, that is the presence of "bad iterations", i.e., iterations where we cannot show good progress. This happens when we are forced to take a short step along the search direction to guarantee feasibility of the iterate.  The number of short steps typically needs to be upper bounded in the convergence analysis with "ad hoc" arguments (see, e.g., \cite{freund2017extended} and \cite{lacoste2015global}). The main idea behind our method is to chain several short steps by skipping gradient updates until proper conditions are met. 

\subsection{Related work}
\textbf{FW variants.} The main drawback of the classic FW algorithm is its slow $O(1/k)$ convergence rate for convex objectives. This rate is tight even for strongly convex objectives on polytopes, due to a well understood zig-zagging behaviour near optima on the boundary (see, e.g., \cite{canon1968tight} and \cite{wolfe1970convergence}). The study of assumptions and variants leading to faster rates is a rapidly developing field. \\ 
Alternative or modified directions moving away from "bad" vertices or atoms have a long history, starting at least with the work of Wolfe \cite{wolfe1970convergence} (see \cite{kolmogorov2020practical}  and \cite{lacoste2015global} for recent references). In addition to considering new directions, the works \cite{braun2019blended} and \cite{braun2017lazifying} propose strategies to skip the linear minimization oracle (LMO) computation from time to time by caching linear minimizers, while the recent work \cite{kolmogorov2020practical} for optimization on polytopes applies recursively a FW variant to smaller polytopes. However, to our knowledge, no strategy to avoid short steps has been discussed in these previous works. \\ 
For smooth strongly convex objectives, the convergence rates of many of these "improved directions" FW variants is linear on polytopes (see, e.g., \cite{beck2017linearly} and \cite{lacoste2015global}). Furthermore, in \cite{kerdreux2019restarting} it was proved that  convergence rate of an AFW variant is adaptive to H\"{o}lderian error bound conditions interpolating between the general convex case and the strongly convex one. \\
A different approach, adopted in the general smooth convex setting, is to use FW variants to approximate projections. In particular, the conditional gradient sliding method uses the FW method to approximate projections on the feasible set within a projected gradient scheme (see, e.g., \cite{hazan2016variance} and \cite{lan2016conditional}). Another approach introduced in \cite{combettes2020boosting} for smooth convex objectives implicitly uses the Non Negative Matching Pursuit (NNMP) algorithm to compute an approximate projection of the negative gradient on the tangent cone. To our knowledge, however, conditional gradient sliding approaches always lead to a sublinear $O(1/\varepsilon)$ LMO complexity, and the approach in \cite{combettes2020boosting} does not lead to any improvement on the $O(1/\varepsilon)$ worst case gradient complexity of the classic FW. \\
 Outside the projection free setting, in \cite{mortagy2020walking} a procedure making multiple steps without updating the gradient (in a fashion similar to our SSC) is defined, and it is claimed that the approach traces the piecewise linear projection curve on polytopes, thus leading to the same linear convergence rate of the standard projected gradient method in the strongly convex setting. \\
In the non convex setting, for the classic FW algorithm a convergence rate of $O(1/\sqrt{k})$ was proved in \cite{lacoste2016convergence} and then extended to other variants in \cite{bomze2020active} and \cite{qu2018non}. \\

\textbf{KL property.} The KL property (see, e.g., \cite{attouch2010proximal}, \cite{bolte2007clarke} and \cite{bolte2010characterizations}) has been extensively applied to compute the convergence rates of proximal subgradient type methods (see, e.g., \cite{attouch2010proximal}, \cite{attouch2013convergence},  \cite{bolte2017error}, \cite{wang2019global} and \cite{xu2013block}). Furthermore, for convex objectives, it has been proved that H\"{o}lderian error bound conditions are a particular case of this property \cite{bolte2017error}. However, we are not aware of previous applications to the Frank-Wolfe variants under study in this paper. \\

\textbf{Angle condition.} The analysis of unconstrained descent methods often relies on some version of an angle condition, imposing an upper bound on the angle between the negative gradient and the descent direction selected by the method (see, e.g.,  \cite{absil2005convergence},  \cite{grippo1986nonmonotone} and \cite{zhang2006descent}). However, due to the presence of short steps and full FW steps, these analyses do not extend to our setting in a straightforward way.\\ In Section 3, we present an angle condition for optimization over a convex set. While to our knowledge this extension is novel for first order optimization methods, analogous conditions can be found in the context of direct search methods for linearly constrained derivative free optimization (see, e.g., \cite{kolda2007stationarity} and \cite{lewis2007implementing}), imposed on the smallest angle between the negative gradient and a search direction. Finally, we remark that our condition was somehow used, but not stated explicitly, in \cite{beck2017linearly} and \cite{lacoste2015global} within the context of smooth strongly convex optimization over polytopes.

\subsection{Contributions}

Our main contributions are twofold:
\begin{itemize}
	\item We formulate an angle condition for projection free methods, and prove that it leads to linear convergence in the number of "good steps" for non convex objectives  satisfying a KL inequality. We show that this condition applies to the away step Frank-Wolfe (AFW), the pairwise Frank-Wolfe (PFW)  and the FW method with in face directions (FDFW) (see, e.g., \cite{freund2017extended}, \cite{lacoste2015global}, \cite{garber2016linear} and \cite{guelat1986some}) on polytopes. First, we give linear rates for good steps in Proposition \ref{p:taurate}. Then, we give global asymptotical rates under the assumption that the number of bad steps between two good steps is bounded in Proposition \ref{p:simpleasymptotic}. We apply this result to FW variants in Corollary \ref{lr:goodsteps}.
	\item We define the SSC procedure, which  can be applied to all the FW variants listed in the first point, and show that it gets improvements on known rates (see Table \ref{tab:2} in Section \ref{Section:SSC}). In particular, we prove that it leads to global linear convergence rates with no bad steps (see Lemma \ref{p:crkl} and Corollary \ref{cor:klrate}) under a global KL inequality and the angle condition. We then prove that we have local linear convergence rates and asymptotical linear convergence rates under a local KL property as well (see Theorem~\ref{t:loc} and Corollary \ref{cor:glob}). This, to our knowledge, is the first (bad step free) linear convergence rate  for FW variants under the KL inequality. In the general smooth non convex case, we further prove, under the angle condition, a $O(1/\sqrt{k})$ convergence rate with respect to a  specific measure of non-stationarity for the iterates, that is the projection of the negative gradient on the convex cone of feasible directions (see Theorem \ref{cor:stationary}, Corollary \ref{cor:gennc} and Remark \ref{r:ncrates}). 
\end{itemize}  
While here we apply our framework only to the AFW, the PFW, and the FDFW on polytopes, we remark that our results hold for projection free methods on generic convex sets. In an extended version of this paper \cite{rinaldi2020unifying} we show applications  on convex sets with smooth boundary for FW variants and methods using orthographic retractions (see also \cite{absil2012projection}, \cite{balashov2019gradient}, \cite{levy2019projection} and references therein). \\
The reasons why eliminating bad steps truly makes a difference in our context are the following: 
\begin{itemize}
	\item it rules out impractical convergence rates due to a large number of bad steps.
	An interesting example is given by
	the rate guarantee reported in \cite{lacoste2015global} for the pairwise Frank-Wolfe (PFW) variant on the $N - 1$ dimensional simplex. This guarantee is indeed more loose than for the other variants, because there is no satisfactory bound on the number of such problematic steps
	(there is a best known bound of $3N!$ bad steps for each good step);
	\item it eliminates the dependence of the convergence rates on the support of the starting point (see, e.g.,  \cite{johnell2020frank} and \cite{kolmogorov2020practical}). This dependence can significantly affect the performance of FW variants on smooth non convex optimization problems \cite{cristofari2020active}.
\end{itemize}

Finally, while beyond the scope of this paper, we mention that bad steps lead to a slow active set identification for the AFW, when compared to the "one shot" identification property characterizing proximal gradient methods and active set strategies (see \cite{cristofari2020active}, \cite{nutini2019active} and references therein). More precisely, analyses in recent works (\cite{bomze2019first}, \cite{bomze2020active} and \cite{garber2020revisiting}) show that a number of bad steps equal to the number of "wrong" atoms is performed by the method in a sufficiently small neighborhood of a solution to identify its support.
\subsection{Paper structure}

The structure of the paper is as follows. In Section 2, we define some notation and state some preliminary results from convex analysis. In Section 3, we introduce the angle condition for first-order projection free methods, show examples of FW variants
satisfying the condition and prove linear convergence in the number of good steps. We define the SSC procedure in Section 4, where we also state the main convergence results. Preliminary numerical results are reported in Section \ref{REST}, while the missing proofs can be found in the appendix. 

\section{Notation and preliminaries}\label{s:prel} 
We consider the following constrained optimization problem:
\begin{equation}\label{eq:mainpb}
	\min \left\{f(x) \ \vert \ x \in \Omega \right\} \, .
\end{equation}
In the rest of the article $\Omega$ is a compact and convex set and $f \in C^1(\OM)$ with $L$-Lipschitz gradient:
$$ \n{\nabla f(x) - \nabla f(y)} \leq L\n{x - y} \tx{ for all }x, y \in \OM \, .$$
We define $D$ as the diameter of $\OM$, $\hat{c} = c/\|c\|$ for $c \in \mathbb{R}^n / \{ 0 \} $ and $\hat{c} = 0$ for $c = 0$. For sequences we write $\{x_k\}$ instead of $\{x_k\}_{k \in I}$  when $I$ is clear from the context, with $[j : i] = \{j, j+1, ..., i-1, i\}$. 	
For $a, b \in \RR \cup \{\pm \infty\}$ we denote as $[a < f(x) < b]$ the set $\{ x \in \OM \ \vert \ f(x) \in (a, b) \}$, with analogous definitions for non strict inequalities. 
For subsets $C, D$ of $\RR^n$ we define $\dist(C, D)$ as
\begin{equation*}
	\dist(C, D) = \inf \{\n{y-z} \ \vert \ z \in C, \ y \in D \} \, ,
\end{equation*}
$B_R(C)$ as the neighborhood $\{x \in \RR^{n} \ \vert \ \dist(C, x) < R \}$  of $C$ of radius $R$ and in particular $B_R(x)$ as the open euclidean ball of radius $R$ and center $x$. When $C$ is closed and convex we define as $\pi(C, \cdot)$ the projection on $C$. If $C$ is a cone then we denote with $C^*$ its polar. 

We now state some elementary properties related to the tangent and the normal cones, where for $\xx\in \Omega$ we denote with $T_{\Omega}(\bar{x})$ and $N_{\OM}(\bar{x})$ the tangent and the normal cone to $\Omega$ in $\xx$ respectively. The next proposition (from \cite{rockafellar2009variational}, Theorem 6.9) characterizes these cones for closed convex subsets of $\RR^n$. 
\begin{proposition} \label{normalC}
	Let $\Omega$ be a closed convex set. For every point $\bar{x}\in \Omega$ we have
	\begin{align*}
		&T_{\Omega}(\bar{x}) = \textnormal{cl}\{w \ \vert \ \exists \lambda>0 \textnormal{ with } \bar{x}+\lambda w \in \Omega \} \, , \\
		\textnormal{int}&(T_{\Omega}(\bar{x})) = \{w \ \vert \ \exists \lambda > 0 \textnormal{ with } \bar{x}+\lambda w \in \textnormal{ int}(\Omega)\} \, , \\ 
		& N_{\Omega}(\bar{x}) = T_{\Omega}(\bar{x})^*=\{v\in \RR^n\ \vert \ (v,y - \xx)\leq 0 \ \forall\ y \in  \OM \} \, .
	\end{align*}	
\end{proposition}
We have the following formula connecting the supremum of a linear function "slope" along feasible directions to the tangent and the normal cone: 
\begin{proposition} \label{NWessential}
	If $\Omega$ is a closed convex subset of $\RR^n$, $\bar{x} \in \Omega$ then for every $g \in \RR^n$
	$$	\max \left \{0, \sup_{h\in\Omega\sm \{\bar{x}\}} \left (g, \frac{h-\bar{x}}{\|h - \bar{x}\|}\right )\right \} = \textnormal{dist}(N_{\Omega}(\bar{x}), g) = \|\pi(T_{\Omega}(\bar{x}), g)\| \, .$$	
\end{proposition}

This property is a consequence of the Moreau-Yosida decomposition \cite{rockafellar2009variational} and we refer the reader to the Appendix for a detailed proof. On polytopes, a geometric interpretation is that the smallest angle between $g$ and a descent direction $d$ feasible in $\bar{x}$ is achieved for $d = \pi(T_{\OM}(\xx), g)$. \\
In the rest of the article to simplify notations we often use $\pi_{\xx}(g)$ as a shorthand for $\|\pi(T_{\Omega}(\bar{x}), g)\|$. Then, by Proposition~\ref{NWessential}, first order stationarity conditions in $\xx$ for the gradient $-g$ become equivalent to $\pi_{\xx}(g) = 0$. \\
In the computation of the convergence rates, we often make the following assumption.
\begin{assumption} \label{ass:KL}
Given a stationary point $x^* \in \OM$, there exists $\eta, \delta > 0$ such that for every $x \in 
[f(x^*) < f < f(x^*) + \eta] \cap B_{\delta}(x^*)$
\begin{equation} \label{hp:klP}
	\pi_x(-\nabla f(x)) \geq \ck (f(x) - f(x^*))^{\frac{1}{2}} \, .
\end{equation}
\end{assumption}

 We refer the reader to the extended version \cite{rinaldi2020unifying} of this article for a study of convergence rates under a more general inequality, interpolating between \eqref{hp:klP} and the generic non convex case. Let now $i_{\OM}$ be the indicator function of $\OM$ so that $i_{\OM}(x) = 0$ in $\OM$ and $i_{\OM}(x) = + \infty$ otherwise. It can easily be seen that \eqref{hp:klP} is a special case of the KL inequality (see, e.g., \cite{attouch2010proximal}, \cite{attouch2013convergence} and \cite{bolte2017error}) with exponent $\frac{1}{2}$ 
\begin{equation}
	\dist(0, \partial f_{\OM}(x)) \geq \ck (f_{\OM}(x) - f_{\OM}(x^*))^{\frac{1}{2}}
\end{equation}
for $f_{\OM} = f + i_{\OM}$, using that
\begin{equation} \label{eq:pxsub}
	\pi_x(-\nabla f(x))  = \dist(-\nabla f(x), N_{\OM}(x)) = \dist(0, \partial(f + i_{\OM})(x))  \, ,
\end{equation}
with the last equality following by Proposition \ref{NWessential}. For convex objectives, condition \eqref{hp:klP} is therefore implied by the Holderian error bound $f(x) - f(x^*) \geq \gamma \dist(x, \mathcal{X}^*)^2$, for $ \mathcal{X}^*$ set of solutions of Problem \eqref{eq:mainpb} (see  \cite[Corollary 6]{bolte2017error}), which in turn is implied by $\mu - $ strong convexity (see, e.g., \cite{karimi2016linear}). Under suitable assumptions (see Proposition \ref{p:PLiKL}) our KL condition is also implied by the classic Polyak-Lojasiewicz inequality $\n{\nabla f(x)} \geq \ck (f(x) - f(x^*))^{\frac{1}{2}}$ (from  \cite{lojasiewicz1963propriete} and \cite{polyak1963gradient}).  Finally, Assumption \ref{ass:KL} is implied by the Luo Tseng error bound \cite{luo1993error} under some mild separability conditions for stationary points (see \cite[Theorem 4.1]{li2018calculus}). This error bound is known to hold in a variety of convex and non convex settings (see Section \ref{s:examples} and references in \cite{li2018calculus}).

\section{An angle condition} \label{SBLB}
Let $\mathcal{A}$ be a first-order optimization method defined for smooth functions on a closed subset $\OM$ of~$\RR^n$. We assume that given first-order information $(x_k, \nabla f(x_k))$ the method always selects $x_{k + 1}$ along a feasible descent direction, so that for $(x, g) \in \OM \times \RR^n$ we can define 
\begin{equation*}
	\Ab(x,g) \subset T_{\OM}(x) \cap \{y \in \RR^n \ \vert \ \s{g}{y} > 0 \} \cup \{0\} \,
\end{equation*}
as the possible descent directions selected by $\Ab$ when $x= x_k$, $g = -\nabla f(x_k)$ for some $k$ (see Algorithm~\ref{tab:1}). When $x$ is first-order stationary, we set $\Ab(x, g) = \{ 0 \}$, otherwise we always assume $0 \notin \Ab(x, g)\neq \emptyset  $.   
\begin{center}
	\begin{table}[h]
		\caption{First-order method}
		\label{tab:1}
		\begin{center}
			\begin{tabular}{l}
				\hline\noalign{\smallskip} 
				\textbf{Initialization.} $x_0 \in \OM$, $k := 0$. \\		
				1. If $x_k$ is stationary, then STOP  \\
				2. select a descent direction  $d_k \in \Ab(x_k, - \nabla f(x_k))$ \\
				3. set $x_{k+1} = x_k + \alpha_k d_k$ for some stepsize $\alpha_k \in [0, \alpha_k^{\max}]$ \\
				4. set $k := k+1$, go to Step 1. \\
				\noalign{\smallskip} \hline
			\end{tabular}
		\end{center}	
	\end{table}
\end{center}

We want to formulate an angle condition for the descent directions selected by $\mathcal{A}$, with respect to the infimum of the angles achieved with feasible descent directions.
In order to do that, we define the directional slope lower bound as
\begin{equation*}
	\tx{DSB}_{\Ab}(\OM, x, g) = \inf_{d \in \Ab(x,g)} \frac{\s{g}{d}}{\pi_x(g) \n{d}} 
\end{equation*}
if $0 \notin \Ab(x,g)$.  Otherwise $x$ is stationary for $-g$, $\pi_x(g) = 0$ and we set $\tx{DSB}_{\mathcal{A}}(\OM, x, g) = 1$. Then with this definition it immediately follows $	\tx{DSB}_{\Ab}(\OM, x, g) \leq 1$ by Proposition \ref{NWessential}. Notice also that when $x \in \tx{int}(\OM)$ then $\tx{DSB}_{\Ab}(\OM, x, g)$ is simply a lower bound on $\cos (\theta_{g, d})$ with $\theta$ the angle between $g$ and a descent direction $d$:
\begin{equation} \label{eq:acunc}
	\tx{DSB}_{\Ab}(\OM, x, g) = \inf_{d \in \Ab(x,g)} \frac{\s{g}{d}}{\n{g} \n{d}} 
\end{equation}
 and thus imposing $\tx{DSB}_{\Ab}(\OM, x, g) \geq \tau$ we retrieve the angle condition \cite[equation (20)]{absil2005convergence}.  We remark that the RHS of \eqref{eq:acunc} defining the unconstrained angle condition is also considered in the constrained setting in \cite{combettes2020boosting} (referred to as alignment condition), as a tool to evaluate potential descent directions. However, without $\pi_{x}(g)$ in the denominator no uniform lower bound can be given for the RHS, and therefore no worst case linear convergence rate (the rate given in \cite[Corollary 3.6]{combettes2020boosting} is in fact $O(1/k)$). \\ 
 Given a subset $P$ of $\OM$ we can finally define the slope lower bound
\begin{equation*}
	\tx{SB}_{\Ab}(\OM, P) = \inf_{\substack{g \in \RR^n \ x \in P}} \tx{DSB}_{\Ab}(\OM, x, g) = \inf_{\substack{g: \pi_x(g) \neq 0 \\ x\in P}} \tx{DSB}_{\Ab}(\OM, x, g)  \, .
\end{equation*}
For simplicity if $P = \OM$ we write $\tx{SB}_{\Ab}(\OM)$ instead of $\tx{SB}_{\Ab}(\OM, \OM)$. \\

We now show a few examples of Frank-Wolfe variants satisfying the following \emph{angle condition}
\begin{equation}\label{eq:ang_cond}
	\tx{SB}_{\Ab}(\OM)=\tau>0,
\end{equation}
i.e. cases where  the slope lower bound is strictly greater than 0. \\

\subsection{Frank-Wolfe variants over polytopes and the angle condition} \label{s:FWv}
We now consider the AFW, PFW and FDFW and show that the angle condition is satisfied when $\Omega$ is a polytope.
The AFW and PFW depend on a set of "elementary atoms" $A$ such that $\OM = \conv(A)$. Given $A$, for a base point $x \in \OM$ we can define
$$S_x = \{S \subset A \ \vert \ x \tx{ is a proper convex combination of all the elements in }S \} \, ,$$ the family of possible active sets for $x$. In the rest of the article $A$ is always clear from the context and for simplicity we write $\tx{PFW}$, $\tx{AFW}$ instead of $\tx{PFW}_A$, $\tx{AFW}_A$. For $x \in \OM$,  $S \in S_x$, $d^{\tx{PFW}}$ is a PFW direction with respect to the active set $S$ and gradient $-g$ iff
\begin{equation} \label{eq:PFW}
	d^{\tx{PFW}} = s - q \textnormal{ with } s \in \argmax_{s \in \OM} \s{s}{g} \textnormal{ and } q \in \argmin_{q \in S} \s{q}{g} \, .
\end{equation}
Similarly, given $x \in \OM$, $S \in S_x$, $d^{\tx{AFW}}$ is an AFW direction with respect to the active set $S$ and gradient $-g$ iff
\begin{equation} \label{AFWdir}
	d^{\tx{AFW}} \in \argmax \{\s{g}{d} \ \vert \ d \in \{d^{\tx{FW}}, d^{AS}\} \} \, ,
\end{equation}
where $d^{\tx{FW}}$ is a classic Frank-Wolfe direction
\begin{equation} \label{eq:FWstep}
	d^{\tx{FW}}  = s - x  \textnormal{ with } s \in \argmax_{s \in \OM} \s{s}{g} \, ,
\end{equation}
and $d^{\tx{AS}}$ is the away direction 
\begin{equation} \label{eq:ASstep}
	d^{\tx{AS}}  = 	x-q  \textnormal{ with } q \in \argmin_{q \in S} \s{q}{g} \,  .
\end{equation}

The FDFW from \cite{freund2017extended}, \cite{guelat1986some} (sometimes referred to as Decomposition invariant Conditional Gradient (DiCG) when applied to polytopes 
\cite{garber2016linear}, \cite{bashiri2017decomposition}) relies only on the current point $x$ and the current gradient $-g$ to choose a descent direction and, unlike the AFW and the PFW, does not need to keep track of the active set. 

The in face direction is defined as
\begin{equation*}
	d^{F} = x_k - x_{F} \tx{ with } x_{F} \in \argmin\{\s{g}{y} \ \vert \ y \in \F(x) \} 
\end{equation*}
for $\F(x)$ the minimal face of $\OM$ containing $x$. 
The selection criterion is then analogous to the one used by the AFW:
\begin{equation} \label{crit:FD}
	d^{\tx{FD}} \in \tx{argmax} \{ \s{g}{d} \ \vert \ d \in \{d^{F}, d^{\tx{FW}} \} \} \, .
\end{equation}
We write $\SB_{\tx{FD}}, \DSB_{\tx{FD}}$ instead of $\SB_{\tx{FDFW}}, \DSB_{\tx{FDFW}}$ in the rest of the paper. When $\OM$ is a polytope and $\vert A\vert  < \infty$, the angle condition holds for the directions and the related FW variants we introduced. Before stating a lower bound for $\tx{SB}_{\Ab}(\OM)$ in this setting we need to recall the pyramidal width constant $\PWidth(A)$
 introduced in \cite{lacoste2015global}. We refer the reader to  \cite{rademacher2020smoothed} and references therein for a discussion of various properties of this and related parameters. 

We use here a characterization of $\PWidth(A)$ proved in \cite{pena2018polytope}:
	\begin{equation} \label{eq:penaPW}
		\PWidth(A) = \min_{\F \in \pfaces({\OM})} \dist(\F , \tx{conv}(A \sm \F)) \, ,
	\end{equation}
	with $\pfaces(\OM)$ the set of proper faces of $\OM$.
	We now introduce one key property of $\tx{PWidth}(A)$ which relates it to the angle along the PFW direction.  While we give a self contained proof of the lemma relying only on \eqref{eq:penaPW}, we remark that the lemma can also be proved using \cite[Theorem 3]{lacoste2015global}.
\begin{lemma} \label{l:PWidth}
	We have the following lower bound
		\begin{equation*} 
			\frac{\s{g}{d^{\tx{PFW}}}}{\n{\pi(T_{\OM}(x), g)}} \geq \textnormal{PWidth}(A) \, .
	\end{equation*}
\end{lemma}
\begin{proof}
	We use $s, q$ and $S$ as in \eqref{eq:PFW}. For $z$ in $\OM$ and $d$ feasible direction in $z$ we define as $\hat{\alpha}^{\max}(z, d)$ the maximal feasible stepsize in the direction $d$. 
		Let $p = \pi(T_{\OM}(x), g)$, and let $y$ be a maximizer of $\hat{\alpha}^{\max}(y, p)$ for $y \in S$. We have
		\begin{equation}
			\begin{aligned}
				&	\Sc{g}{d^{\tx{PFW}}} = \Sc{g}{(s - y) + (y - q)} \geq \Sc{g}{s - y} \geq \Sc{g}{(y + \hat{\alpha}^{\max}(y, p)p) - y} \\ 
				\geq & \frac{\PWidth(A)}{\n{p}} \Sc{g}{p} =  \PWidth(A) \n{p} \, ,  	
			\end{aligned}
		\end{equation} 
		where we used Lemma \ref{l:maxstepsize} in the third inequality, and $ \Sc{g}{p} = \n{p}^2$ as it follows by the Moreau-Yosida decomposition in the last equality.
\end{proof}

In order to define an angle condition for the FDFW, we use the following upper bound on $\PWidth(A)$, independent from the particular set $A$ chosen to represent $\OM$:
	\begin{equation}
		\PFWidth(\OM) = \min_{\substack{\F_1, \F_2 \in \pfaces(\OM) \\ \F_1 \cap \F_2 = \emptyset}} \dist(\F_1, \F_2) \, .
	\end{equation}
\begin{proposition}\label{eqphi}
	$
	\textnormal{SB}_{\tx{PFW}}(\Omega) \geq \tau_p := \frac{\PWidth(A)}{D}  \, , 
	\textnormal{SB}_{\tx{AFW}}(\Omega) \geq \frac{\tau_p}{2} \, , 
	\textnormal{SB}_{\tx{FD}}(\Omega)  \geq \frac{\tau_v}{2} := \frac{\PFWidth(\OM)}{2D} \, . $		
\end{proposition} 
\begin{proof}
	Let $g$ be such that $\pi_x(g) \neq 0$. We have
	\begin{equation*}\label{eq:PFWIN}
		\begin{aligned}
			& \DSB_{\tx{PFW}}(\OM, x, g) = \inf_{d^\tx{PFW} \in \tx{PFW}(x, g)} \frac{\s{g}{d^{\tx{PFW}}}}{\n{d^{\tx{PFW}}}\n{\pi(T_{\OM}(x), g)}} \\
			& \geq \frac{\s{g}{d^{\tx{PFW}}}}{D\n{\pi(T_{\OM}(x), g)}} \geq \frac{\textnormal{PWidth}(A)}{D} \, ,	
		\end{aligned}
	\end{equation*}
	where we used Lemma \ref{l:PWidth} in the last inequality. \\
	Hence $\textnormal{SB}_{\tx{PFW}}(\Omega) \geq \frac{\PWidth(A)}{D}$ follows by taking the inf on the LHS for $x \in \OM$ and $g$ such that $\pi_x(g) \neq 0$ in \eqref{eq:PFWIN}.
	The inequality $\textnormal{SB}_{\tx{AFW}}(\Omega) \geq \frac{\PWidth(A)}{2D}$  is a corollary since
	\begin{equation*}
		\s{g}{d^{\tx{AFW}}} \geq \frac{1}{2}\s{g}{d^{\tx{PFW}}} \, ,
	\end{equation*}
	as it follows immediately from the definitions (see also \cite[equation (6)]{lacoste2015global}). \\		
	The angle condition for the FDFW can be proved analogously to the angle condition for the AFW, where in Lemma \ref{l:maxstepsize} the RHS can be improved with $\PFWidth(\OM)$ instead of $\PWidth(A)$ using that the active set $A'$ can be taken as the set of vertices of a face.
\end{proof}	

\begin{remark} \label{r:1}
	Results analogous to the ones in Proposition \ref{eqphi} can be proven relatively to the vertex facial distance $\tx{vf}(\OM)$ from \cite{beck2017linearly}. More precisely, assuming $A = V(\OM)$, for $V(\OM)$ set of vertices of $\OM$, and that the AFW and the PFW keep active sets of size at most $\bar{s}$, we have $\textnormal{SB}_{\tx{PFW}}(\Omega) \geq \frac{\tx{vf}(\OM)}{\bar{s}D}$, $\textnormal{SB}_{\tx{AFW}}(\Omega) \geq \frac{\tx{vf}(\OM)}{2\bar{s}D}$ as a consequence of \cite[Lemma 3.1]{beck2017linearly}. Furthermore, for the FDFW we have $\SB_{\tx{FD}}(\OM, \OM_{\bar{s}}) \geq \frac{\tx{vf}(\OM)}{2\bar{s}D}$, with $x \in \Omega_{\bar{s}} \subset \Omega$ iff there exists $S \in S_x$ such that  $\vert S\vert  \leq \bar{s}$.
\end{remark}

\subsection{Linear convergence for good steps under the angle condition}
Consider now a method following the scheme described by Algorithm \ref{tab:1} and with stepsize given by 
\begin{equation} \label{def:Lstep}
	\alpha_k = \min \left( \bak, \alpha_k^{\max} \right) \, ,
\end{equation}
where
\begin{equation}\label{eq:step0}
	\bak= \frac{\s{-\nabla f(x_k)}{d_k}}{L \n{d_k}^2}.
\end{equation}
We notice that $\bar \alpha_k$ in \eqref{eq:step0} is a standard stepsize, often used in numerical tests with a properly tuned estimate for $L$ (see, e.g.,~\cite{pedregosa2020linearly}). The following lemma shows that at every iteration a sufficient decrease condition is satisfied, independently from the method $\mathcal{A}$, when using stepsize \eqref{eq:step0}.

\begin{lemma} \label{l:suffdec}
If $\alpha_k \leq \bak$, thus in particular for the stepsize \eqref{def:Lstep}, we have:
 		\begin{equation} \label{eq:suffdec}
 	f(x_{k}) - f(x_{k + 1}) \geq \frac{L}{2}\n{x_k - x_{k + 1}}^2 \, .
 \end{equation}
\end{lemma}
The proof is straightforward and we defer it to the appendix. \\	
Assume now that the method $\mathcal{A}$ used by Algorithm \ref{tab:1} satisfies the angle condition~\eqref{eq:ang_cond}. We say that the algorithm performs a \emph{full FW step} if 
\begin{equation}\label{eq:fullFWs}
	x_{k + 1} \in \argmin_{x \in \OM}\s{\nabla f(x_k)}{x} \, .	
\end{equation}
In the following proposition, we prove a general linear convergence rate in the number of \textit{good steps}, i.e., the steps satisfying $\alpha_k = \bar{\alpha}_k$ or \eqref{eq:fullFWs}, under the assumption that  the method $\Ab$ satisfies the angle condition \eqref{eq:ang_cond}, and that the KL inequality ~\eqref{hp:klP} holds for the objective function $f$ in Problem \eqref{eq:mainpb}. 

\begin{proposition} \label{p:taurate}	
	Let us assume that 	$\Ab$ satisfies the angle condition \eqref{eq:ang_cond}, and the objective function $f$ in Problem \eqref{eq:mainpb} satisfies condition~\eqref{hp:klP} in $x_k$ and $x_{k + 1}$.  
	\begin{itemize}
		\item If $\alpha_k = \bar{\alpha}_k$ then
		\begin{equation} \label{eqt:dec}
			f(x_{k+1}) - f(x^*) \leq \left(	1 - \frac{\mu}{L} \tau^2 \right)(f(x_k) - f(x^*)) \, . 
		\end{equation} 
		\item If  the step $k$ is a full FW step then
		\begin{equation} \label{eqt:dec2}
			f(x_{k+1}) - f(x^*) \leq \left(1 + \frac{\mu}{L} \right)^{-1}  (f(x_k) - f(x^*)) \, .
		\end{equation}
	\end{itemize}
\end{proposition}
\begin{proof}
 	Let  $p_k = \n{\pi(T_{\OM}(x_{ k + 1}), -\nabla f(x_{ k + 1}))}$ and $\tilde{p}_k = \n{\pi(T_{\OM}(x_{ k + 1}), -\nabla f(x_k))}$. We have
	\begin{equation} \label{eq:projineq}
		\begin{aligned}
			\vert p_k - \tilde{p}_k\vert  & = \vert\n{\pi(T_{\OM}(x_{ k + 1}), -\nabla f(x_{ k + 1}))} - \n{\pi(T_{\OM}(x_{ k + 1}), -\nabla f(x_{k}))}\vert  \\ &\leq \n{-\nabla f(x_{ k + 1}) + \nabla f(x_{k})} \leq L\n{x_{ k + 1}-x_k} \, ,
		\end{aligned}		
	\end{equation}
	where we used the 1-Lipschitzianity of projections in the first inequality. \\

	If $\alpha_k = \bar{\alpha}_k$ then
	\begin{equation} \label{eq:full}
		\begin{aligned}
			f(x_{k+1}) = & f(x_k + \bar{\alpha}_k d_k) \leq f(x_k) - \frac{1}{2L}\left( \frac{\s{\nabla f(x_k)}{d_k}}{\n{d_k}} \right)^2  
			\leq f(x_k) - \frac{\tau^2}{2L}p_{k - 1}^2 \\ \leq & f(x_k) - \frac{\mu \tau^2}{L}(f(x_k) - f(x^*)) \, 	,
		\end{aligned}
	\end{equation}
	where we used \eqref{eq:stdc} in the first inequality, $\tx{SB}_{\Ab}^f(\OM) = \tau$ in the second one, and condition \eqref{hp:klP} in the third one.\\
	If the step $k$ is a full FW step then $\tilde{p}_k = 0$ because $x_{k + 1} \in \argmin_{y \in \OM} \Sc{\nabla f(x_k)}{y} \Leftrightarrow -\nabla f(x_k) \in N_{\OM}(x_{k + 1}) \Leftrightarrow \n{\pi(T_{\OM}(x_{k + 1}), -\nabla f(x_k) )} = 0$, where the last equivalence is true by Proposition \ref{NWessential}. Then
	\begin{equation} \label{eq:fullFW}
		f(x_{k + 1}) - f(x^*) \leq \frac{p_k^2}{2\mu} \leq \frac{(\tilde{p}_k + L\n{x_{k+1} - x_k})^2}{2\mu} = \frac{L^2}{2\mu} \n{x_{k + 1} - x_k}^2 \leq \frac{L}{\mu} (f(x_{k}) - f(x_{k + 1})) \, ,
	\end{equation}
	where we used \eqref{hp:klP} in the first inequality, \eqref{eq:projineq} in the second, $\tilde p_k = 0$ and \eqref{eq:sdec} in the last inequality. Then \eqref{eq:sdec} and \eqref{eqt:dec2} follow by rearranging \eqref{eq:full} and \eqref{eq:fullFW} respectively. 
\end{proof}	
We finally report an asymptotic rate under the additional assumption that bad steps between two good steps are limited.
	\begin{proposition} \label{p:simpleasymptotic}
		Assume that the number of bad steps between two good steps is limited and that $\Ab$ satisfies the angle condition \eqref{eq:ang_cond}. Then:
		\begin{itemize}
			\item every accumulation point of $\{x_k\}$ is stationary, and $f(x_k)$ is decreasing and convergent to $f^* \in \RR$;
			\item if Assumption \ref{ass:KL} holds for every stationary point in the level set $[f(x) = f^*]$, we have the asymptotic convergence rate:
			\begin{equation}
				f(x_k) - f(x^*) \leq M \bar{q}^{\bar{\gamma}(k)} \, ,
			\end{equation}
			for some $M > 0$, $\bar{\gamma}(k)$ number of good steps among the first $k$ steps and 
			\begin{equation}
				\bar{q} = \max \left( \left(1 + \frac{\mu}{L} \right)^{-1},	\left(	1 - \frac{\mu}{L} \tau^2 \right)\right) \, .
			\end{equation}
		\end{itemize}
	\end{proposition}
	\begin{proof}
		Let $k(j)$ be the subsequence of iterates associated to good steps, so that by assumption $k(j + 1) - k(j)$ is bounded, and define $\tilde{k}(j) = k(j) - 1$ if $\alpha_{k(j)} = \bar{\alpha}_{k(j)}$, $\tilde{k}(j) = k(j)$ otherwise. Notice that $\tilde{k}(j + 1) - \tilde{k}(j)$ is also bounded. By \eqref{eq:sdec} we have that $\{f(x_k)\}$ is decreasing and thus convergent to $f^* \in \RR$, and also that $\n{x_{k} - x_{k + 1}} \rightarrow 0$. With the notation used in Proposition \ref{p:taurate} we now claim $p_{\tilde{k}(j)} \rightarrow 0$. In fact if $\alpha_{k(j)} = \bar{\alpha}_{k(j)}$ then
		\begin{equation}
			p_{\tilde{k}(j)}^2 = p_{k - 1}^2 \leq \frac{2L}{\tau^2}(f(x_k) - f(x_{k + 1}))  \rightarrow 0 \, ,
		\end{equation}
	    where we used \eqref{eq:full} in the inequality, and if $k(j)$ is a full FW step then 
	    \begin{equation}
	    	p_{\tilde{k}(j)} \leq p_{k(j)} \leq \tilde{p}_{k(j)} + L\n{x_{k(j) + 1} - x_{k(j)}} =L\n{x_{k(j) + 1} - x_{k(j)}} \rightarrow 0 \, , 
	    \end{equation}
    where we used \eqref{eq:projineq} in the first inequality and  $\tilde{p}_{k(j)} = 0$ in the equality. \\
    We therefore have $p_{\tilde{k}(j)} \rightarrow 0$. Equivalently, thanks to \eqref{eq:pxsub} we have $\dist(0, \partial f_{\OM}(x_{\tilde{k}(j)})) \rightarrow 0$, so if $x^*$ is a limit point of $x_{\tilde{k}(j)}$ by lower semicontinuity of the subdifferential we must have $0 \in \partial f_{\OM}(x^*)$, i.e., $x^*$ is stationary. In particular, by compactness $\{x_{\tilde{k}(j)}\}$ must converge to the set of stationary points. By the boundedness of $\n{x_{k + 1} - x_k}$ and $\tilde{k}(j + 1) - \tilde{k}(j)$ we also have that the set of limit points of $\{x_k\}$ coincides with the set of limit points of $\{x_{\tilde{k}(j)}\}$, and in particular it is a subset of stationary points contained in $[f(x) = f^*]$. \\
		 Let $\bar{\OM} \subset [f(x) = f^*]$ be the set of limit points of $\{x_k\}$. By compactness (see \cite[Lemma 6]{bolte2014proximal}), we have that for some fixed $\varepsilon, \eta > 0$, the KL property holds for every $x^* \in \bar{\OM}$ with parameters $\varepsilon$ and $\eta$. Then for $k$ large enough $x_k \in B_{\delta}(x^*) \cap [f(x^*) < f < f(x^*) + \eta ]$ for some $x^* \in \OM$, and the asymptotic rates follow by Proposition \ref{p:taurate}.  
\end{proof}


For the three FW variants described before we can now give an asymptotic linear convergence rate in the number of good steps. We refer the reader to Table \ref{tab:2} for bounds on this number.

\begin{corollary}\label{lr:goodsteps}
	Let us assume that the objective function $f$ satisfies Assumption \ref{ass:KL} for every stationary point in the level set $[f(x) = f^*]$ and $\Omega = \tx{conv}(A)$ with $\vert A\vert  < +\infty $ in Problem \eqref{eq:mainpb}. Then the AFW, the PFW and the FDFW converge at a rate
	\begin{equation} \label{eqt:grate}
		f(x_k) - f(x^*) \leq M\qfs^{\bar{\gamma}(k)},
	\end{equation}
	for some $M >0$, with $\bar{\gamma}(k)$ the number of good steps among the first $k$ steps,
	\begin{equation} \label{eqt:AFW}
		\qfs = \max \left( 1 -  \frac{\mu}{L} \left( \frac{\PWidth(A)}{2D} \right)^2, \left(1 + \frac{\mu}{L}\right)^{-1}\right)	
	\end{equation} 
	for the AFW,	
	\begin{equation} \label{eqt:PFW}
		\qfs = 1 -  \frac{\mu}{L}\left( \frac{\PWidth(A)}{D} \right)^2	
	\end{equation}
	for the PFW,  and 
	\begin{equation} \label{eqt:FDFW}
		\qfs = \max \left( 1 -  \frac{\mu}{L} \left( \frac{\PFWidth(\OM)}{2D} \right)^2, \left(1 + \frac{\mu}{L}\right)^{-1} \right)	
	\end{equation}
	for the FDFW.
\end{corollary}

\begin{proof}
	For the AFW and the FDFW the rates \eqref{eqt:AFW} and \eqref{eqt:FDFW} for good steps follow directly from \eqref{eqt:dec} and \eqref{eqt:dec2} together with the bound on $\tau$ given in Proposition \ref{eqphi}. Since the PFW never performs full FW steps, its rate \eqref{eqt:PFW} for good steps follow directly from \eqref{eqt:dec} together with the bound on $\tau$ given in Proposition \ref{eqphi}.  Finally, given that the number of bad steps between two good steps is limited for all these methods (see \cite{lacoste2015global, kolmogorov2020practical}), we have all the assumptions to apply Proposition \ref{p:simpleasymptotic}.
\end{proof}

\section{First order projection free methods with SSC procedure}\label{Section:SSC}	
We introduce here the SSC procedure, and prove convergence rates both  under the KL inequality~\eqref{hp:klP} and in the generic non convex case.

\subsection{The SSC procedure} \label{s:SSC}
The SSC procedure  chains consecutive short steps, thus skipping updates for the gradient (and possibly for related information, like linear minimizers), until proper stopping conditions are met.
Such a procedure, whose detailed scheme is given in Algorithm~\ref{alg:4}, can be easily embedded in a first-order approach (see Algorithm \ref{alg:3}).

\begin{center}
	\begin{table}[h]
		\caption{First-order method with SSC}
		\label{alg:3}
		\begin{center}
			\begin{tabular}{l}
				\hline\noalign{\smallskip} 
				\textbf{Initialization.} $x_0 \in \OM$, $k = 0$.      \\
				1. \textbf{while} $x_k$ is not stationary: \\
				2. \quad $g = -\nabla f(x_k)$ \\
				3. \quad $x_{k+1} = \textnormal{SSC}(x_k, g)$ \\							
				5. \quad $k = k+1$. \\
				\noalign{\smallskip} \hline
			\end{tabular}
		\end{center}	
	\end{table}
\end{center}  

\begin{center}
	\begin{table}[h]
		\caption{SSC$(\xx, g)$ }
		\label{alg:4}
		\begin{center}
			\begin{tabular}{l}
				\hline\noalign{\smallskip} 
				\textbf{Initialization.} $y_0 = \xx$, $j=0$.  \\	
				\textbf{Phase I} \\
				
				1. \quad select $ d_j \in \Ab(y_j, g)$, $\alpha^{(j)}_{\max} \in \alpha_{\max}(y_j, d_j)$ \\
				2. \quad \textbf{if } $d_j = 0$ \textbf{then:} \\
				3. \quad \quad \textbf{return } $y_j$ \\
				\textbf{Phase II}		\\
				4. \quad compute $\beta_j$ with \eqref{betaj} \\					 		 		
				5. \quad let $\alpha_j = \min(\alpha^{(j)}_{\max}, \beta_j)$ \\					
				6. \quad $y_{j+1} = y_j + \alpha_j d_j$ \\
				7. \quad \textbf{if} $\alpha_j =  \beta_j$  \textbf{then}: \\
				8. \quad \ \ \textbf{return } $y_{j+1}$ \\
				9. \quad $j = j+1$, go to Step 1. \\
				\noalign{\smallskip} \hline
			\end{tabular}
		\end{center}	
	\end{table}
\end{center}
Given that the gradient $-g$ is constant during the SSC, this procedure is an application of $\Ab$ for the minimization of the linearized objective $f_g(z) = \s{ - g}{z - \bar{x}} + f(\bar{x})$ with peculiar stepsizes and stopping criterion. More specifically, after a stationarity check (Phase I), the stepsize $\alpha_j$ is computed by taking the minimum between the maximal stepsize $\alpha^{(j)}_{\max}$ (which we always assume to be greater than $0$) and an auxiliary stepsize $\beta_j$. The point $y_{j + 1}$ generated in Phase II is always feasible since $\alpha_j \leq \alpha^{(j)}_{\max}$ is always smaller than the maximal feasible stepsize along the direction $d_j$. Notice that if the method $\mathcal{A}$ used in the SSC performs a FW step (see equation \eqref{eq:FWstep} for the definition of FW step), then the SSC terminates, with $\alpha_j = \beta_j$ or with $y_{j + 1}$ global minimizer of $f_g$. 

The auxiliary step size $\beta_j$ is defined as the maximal feasible stepsize for the trust region
\begin{equation} \label{eq:omegaaux}
	\OM_j =   \BB_{\n{g}/2L}(\bar{x} + \frac{g}{2L}) \cap \BB_{\s{g}{\hat{d}_j}/L}(\bar{x}) 
\end{equation} 
when $y_j \in \Omega_j$, otherwise the method stops returning $y_j$. Summarizing,
\begin{equation} \label{betaj}
	\beta_j = 
	\begin{cases}
		0 &\tx{ if } y_j \notin \OM_j \, , \\
		\beta_{\tx{max}}(\OM_j, y_j, d_j) & \tx{ if } y_j \in \OM_j \, ,
	\end{cases}
\end{equation}
where $\beta_{\max}(\OM_j, y_j, d_j) = \max \{\beta \in \RR_{\geq 0} \ \vert \ y_j + \beta d_j \in \OM_j \}$ is the maximal feasible stepsize in the direction $d_j$ starting from $y_j$ with respect to $\OM_j$. Since $\OM_j$ is the intersection of two balls there is a simple closed form expression for $\beta_j$. In particular, using that $y_0 = \bar{x}$, if $d_0 \neq 0$ we have
\begin{equation*} 
	\beta_0 = \frac{\s{g}{\hat{d}_0}}{L\n{d_0}}\, ,
\end{equation*}
which corresponds to \eqref{def:Lstep} in the non maximal case, and where $\beta_0 > 0$ since $d_0 \neq 0$ is by assumption a descent direction for $-g$. 

Employing the trust region $\Omega_j$ in the definition of $\beta_j$ guarantees the sufficient decrease condition
\begin{equation} \label{eq:dcondy}
	f(y_j)  \leq f(x_k) - \frac{L}{2}\n{x_k - y_j}^2 
\end{equation}	
and monotonicity of the true objective $f$ during the SSC.  

To see why \eqref{eq:dcondy} holds, notice that the second ball $\bar{B} = \BB_{\n{g}/2L}(x_k + \frac{g}{2L}) $ appearing in the definition of $\OM_j$ does not depend on $j$, so that since $y_0 \in \bar{B}$ we have $y_j \in \bar{B}$ for every $j \in [0:T]$, with $T$ maximal iteration index of the SSC. This is enough to obtain \eqref{eq:dcondy} because for every $z \in \bar{B}$ we have
\begin{equation} \label{stdbarB}			
	f(z) \leq f(\bar{x}) - \s{g}{z-\bar{x}} + \frac{L}{2}\n{z - \bar{x}}^2 \leq f(\bar{x}) - \frac{L}{2}\n{\bar{x} - z}^2	\, ,	
\end{equation}
where the first inequality is the standard descent lemma and the second follows from the definition of $\bar{B}$.

We prove that the true objective $f$ is monotone decreasing in the next lemma.
\begin{lemma} \label{l:decreasing}
	Let us assume $y_j \in \BB_{\s{g}{\hat{d}_j}/L}(\bar{x})$. Then for every $\beta \in [0, \beta_j]$ we have 
	\begin{equation*}
		\frac{d}{d\beta} f(y_j + \beta d_j) \leq 0 \, ,
	\end{equation*}	
	and thus in particular $f(y_j + \beta_j d_j) \leq f(y_j)$.
\end{lemma}
\begin{proof}
	We have 
	\begin{equation*}
		\begin{aligned}
			& \frac{d}{d\beta} f(y_j + \beta d_j) = \n{d_j} \s{\nabla f(y_j + \beta d_j)}{\hat{d}_j}  \\
			= & \n{d_j} \s{ (\nabla f(y_j + \beta d_j) + g) - g}{\hat{d}_j} = \n{d_j} (\s{\nabla f(y_j + \beta d_j) + g}{\hat{d}_j} - \s{g}{\hat{d}_j}) \\ 
			\leq & \n{d_j} (L\n{\bar{x}-y_j - \beta d_j} - \s{g}{\hat{d}_j}) \leq 0 \, ,
		\end{aligned}		 
	\end{equation*}
	where we used $g = -\nabla f(\bar{x})$ and the Lipschitzianity of $\nabla f$ in the first inequality and $$y_j + \beta_j d_j \in \BB_{\s{g}{\hat{d}_j}/L}(\bar{x})$$ in the second. 
\end{proof}

The next result illustrates how the sequence $\{x_k\}$ generated by Algorithm \ref{alg:3} satisfies certain descent conditions. This is an adaptation to our setting of the ones used in the analysis of many proximal type gradient methods (see \cite{attouch2010proximal}, \cite{attouch2013convergence}, \cite{bolte2017error} and references therein). A subtle difference is the introduction of an "hidden sequence" $\{\tilde{x}_k\}$ to control the projection of the negative gradient on the tangent cone. 
\begin{proposition} \label{keyprop}
	Let us consider the sequence $\{x_k\}$ generated by 
	Algorithm \ref{alg:3} and assume  that 
	\begin{itemize}
		\item the angle condition \eqref{eq:ang_cond} holds;
		\item the SSC condition terminates in a finite number of steps.
	\end{itemize}
	Then
	\begin{align}
		f(x_k) - f(x_{k +  1}) & \geq \frac{L}{2}\n{x_k - x_{k +  1}}^2 \, , \label{dq} \\
		\n{x_k - x_{k +  1}} & \geq  \K \n{\pi(T_{\OM}(\tilde{x}_k), -\nabla f(\tilde{x}_k))} \label{dq2}
	\end{align}	      
	for some $\tilde{x}_k \in \{y_j\}_{j=0}^T$ such that $f(x_{k +  1}) \leq f(\tilde{x}_k) \leq f(x_k) - \frac{L}{2} \n{x_k - \tilde{x}_k}^2$,  $\n{\tilde{x}_k - x_k} \leq \n{x_{k + 1} - x_k}$  and for $\K =\tau/(L(1+\tau))$.
\end{proposition}

\subsection{SSC for Frank-Wolfe variants} \label{pfwafw}

	In this section, we show how to apply our results to the PFW, the AFW and the FDFW on polytopes, i.e., we  prove finite termination of the SSC procedure when one of these methods is considered in Algorithm \ref{alg:3}. We also give worst case and average worst case bounds for the number of iterations of the SSC.
We start by proving a general termination criterion.
\begin{lemma} \label{l:termination}
	Assume that the method $\Ab$ applied to any linear function $L_g(x) = - \Sc{g}{x}$ on the feasible set $\OM$ and with every stepsize maximal always terminates in at most $T$ iterations with an optimal solution, i.e. generates a sequence $\{y_j\}_{j \in [0, T']}$ with $T' \leq T$ and $y_{T'} \in \argmin_{x \in \OM} L_{g}(x)$. Then the SSC with the method $\Ab$ on the feasible set $\OM$ always terminates in at most $T$ iterations.
\end{lemma}
\begin{proof}
	Assume by contradiction that the SSC does at least $T + 1$ iterations, generating the sequence $\{y_j\}_{j \in [0:T+1]}$ before terminating. Notice that in this case the SSC must always do maximal steps for $j \in [0:T]$, because it terminates at step 8 when $\alpha_j = \beta_j$ and in particular if $\alpha_j < \alpha^{(j)}_{\max}$. Then for some $T' \leq T$ we must have that $y_{T'} \in \argmin_{x \in \OM} L_{g}(x)$, which gives a contradiction because in this case the method can't find a feasible descent direction in Phase I and terminates returning $y_{T'}$. 
\end{proof}
	\begin{remark} \label{r:finiteT}
	Using the same line of reasoning, it is not difficult to prove that the SSC always terminates if the method $\Ab$ applied to linear objectives and with stepsizes always maximal generates a (possibly finite) sequence $\{y_j\}$ satisfying
	\begin{equation}
		\liminf \pi_{y_j}(g) = 0 \, .
	\end{equation}
\end{remark}

We now denote with $\{S^{(j)}\}$ the sequence of active sets generated by the AFW and the PFW method in the SSC, and with $y_j$ proper convex combination of the elements in $S^{(j)}$. Furthermore, for the FDFW we assume that the maximal stepsize is given by feasibility conditions as in \cite{freund2017extended}:
\begin{equation} \label{eq:alphamax}
	\alpha_{\max}(x, d) = \max \{\alpha \in \RR_{\geq 0} \ \vert \ x + \alpha d \in \OM\} \, .
\end{equation}
Notice that after a maximal in face step from $y_j$ we have $\tx{dim} (\mathcal{F}(y_{j+1})) < \tx{dim}(\mathcal{F}(y_{j}))$ because $y_{j+1}$ lies on the boundary of $\mathcal{F}(y_j)$. 

	\begin{proposition} \label{p:SSCworst}
		The SSC always terminates in at most:
		\begin{itemize}
			\item $\vert A \vert$ iterations for the AFW,
			\item $\vert A \vert - 1$ iterations for the PFW,
			\item $\dim(\OM) + 1$ iterations for the FDFW.
		\end{itemize}
	\end{proposition}
	\begin{proof}
		By Lemma \ref{l:termination} we just need to bound the maximum number of iterations if the method performs always maximal steps for a linear objective $L_g(x)$. The AFW can do at most $\vert A \vert - 1 $ consecutive maximal away steps, since at every such step the number of active atoms decreases by one. Analogously, the FDFW can do at most $\dim(\OM)$ consecutive maximal in face steps, since at every such steps the dimension of the minimal face containing the current iterate decreases by one. The respective bound follows Lemma \ref{l:termination} by noticing that in the linear case the methods terminate after a full FW step. For the PFW, the linearity of the objective implies that only atoms in $\bar{A}:=\argmax_{a \in A}\Sc{g}{x}$ can be added to the support, and only atoms in $A \sm \bar{A}$ can be dropped from the support. In particular, once an atom is dropped from the active set it cannot be added again, and since at every maximal step the PFW drops an atom from the active set its maximal number of iterations is $\vert A \sm \bar{A} \vert \leq \vert A \vert - 1$.
	\end{proof}
	\begin{proposition} \label{p:SSCworstaverage}
		Assume that the linear minimizer is not changed during the SSC. Then, for an infinite sequence $\{x_k\}$, the worst case average number of iterations is
		\begin{itemize}
			\item 2 for the AFW and the PFW,
			\item $\Delta(\OM) + 1$ for the FDFW.
		\end{itemize}
	\end{proposition}
The proof uses analogous arguments to the ones in \cite[Theorem 8]{lacoste2015global} to bound the number of bad steps and we defer it to the appendix. 

\subsection{Convergence rates}
\subsubsection{Smooth non convex objectives}
We first prove, in the generic smooth non convex case, convergence to the set of stationary points with a rate of $O(\frac{1}{\sqrt{k}})$ for $\n{\pi(T_{\OM}(\tilde{x}_i), - \nabla f (\tilde{x}_i))}$.
\begin{theorem} \label{cor:stationary}
	Let us  consider the sequence $\{x_k\}$ generated by  Algorithm \ref{alg:3} and assume that  
	\begin{itemize}
		\item the angle condition \eqref{eq:ang_cond} holds;
		\item 	the SSC procedure always terminates in a finite number of steps.
	\end{itemize}		
	Then  $\{f(x_k)\}$ is decreasing, $f(x_k) \rightarrow~\tilde{f} \in \mathbb{R}$ and the limit points of $\{x_k\}$ are stationary. Furthermore, for any sequence $\{\tilde{x}_k\}$ satisfying the conditions of Proposition \ref{keyprop}, we have $\n{\tilde{x}_k - x_k} \rightarrow 0$, and
	\begin{equation} \label{eq:qsqrk}
		\min_{0 \leq i \leq k} \n{\pi(T_{\OM}(\tilde{x}_i), - \nabla f (\tilde{x}_i))} \leq \min_{0 \leq i \leq k} \frac{\n{x_{i + 1} - x_i}}{\K} \leq  \sqrt{ \frac{2(f(x_0) - \tilde{f})}{\K ^2L(k + 1)}} \, ,
	\end{equation}
	for $\K = \tau/ (L(1 + \tau))$.
\end{theorem}
We now give a corollary for Theorem \ref{cor:stationary} specialized to the FW variants described in Section \ref{s:FWv}  (see also Table \ref{tab:5}). 

\begin{corollary} \label{cor:gennc}
	Let us  assume that $\Omega = \tx{conv}(A)$, with $\vert A\vert  < +\infty $ in Problem \eqref{eq:mainpb}. Then the sequence $\{x_k\}$  generated by  Algorithm \ref{alg:3} with AFW (PFW or  FDFW) in the SSC
	converges at a rate given by equation~\eqref{eq:qsqrk}, with $\tau = \tau_p/2$ ($\tau_p$ or $\tau_v/2$, respectively).
\end{corollary}
\begin{proof}
	Finite termination of the SSC follows by Proposition \ref{p:SSCworst}, and the angle condition is satisfied by Proposition \ref{eqphi}. Thus we have all the assumptions to apply Theorem \ref{cor:stationary}.
\end{proof}
\setcounter{table}{1}
\captionsetup[table]{name=Table, labelfont=bf}
\begin{table}
	\centering
	\bgroup
	\def\arraystretch{2.1}
	\scalebox{0.78}{\begin{tabular}{|l|c|c|c|c|c|} 
			\hline
			Algorithm & Article & LMO c.r.       & Gradient c.r.  & Gap                   \\
			\hline
			NCGS               & \cite{qu2018non}                & $O\left(\frac{1}{k^{0.25}}\right)$ & $O\left(\frac{1}{\sqrt{k}}\right)$ & $\min_{0 \leq i \leq k}\pi(x_i)$ \\ \hline
			AFW, FW            & \cite{bomze2020active}, \cite{lacoste2016convergence}                & $O\left(\frac{1}{\sqrt{k}}\right)$ & $O\left(\frac{1}{\sqrt{k}}\right)$ & $\min_{0 \leq i \leq k}G(x_i)$                     \\ \hline
			AFW, PFW, FDFW + SSC     & Ours                & $O\left(\frac{1}{\sqrt{k}}\right)$                    & $O\left(\frac{1}{\sqrt{k}}\right)$ & $\min_{0 \leq i \leq k} \n{\pi(T_{\OM}(\tilde{x}_i), - \nabla f (\tilde{x}_i))}$  \\           \hline              
	\end{tabular}}
	\egroup
	\caption{Comparison between convergence rates in the generic smooth non convex case. See also Remark \ref{r:ncrates}. 
		$\pi(x) = \n{x - \pi\left(\OM, x - \frac{\nabla f(x)}{2L}	\right)}$, $G$ is the FW gap (see \eqref{def:FWgap}).}
	\label{tab:5} 
\end{table}

\begin{remark} \label{r:ncrates}
	Let $G: \OM \rightarrow \RR_{\geq 0}$ be the FW gap (see, e.g., \cite{lacoste2016convergence}): 
	\begin{equation} \label{def:FWgap}
		G(x) = \max_{s \in \OM} \s{-\nabla f(x)}{s-x} \, .
	\end{equation} 
	Then, for any $y \in \OM$ 
	\begin{equation}\label{eq:gapsin}
		G(y) = \max_{s \in \Omega} \s{-\nabla f(y)}{s - y} = \max_{s \in \Omega \sm \{y\}} \n{s - y} \s{-\nabla f(y)}{\frac{s - y}{\n{s - y}}} \leq D \n{\pi (T_{\OM}(y), -\nabla f(y))} \, ,
	\end{equation}
	where the inequality follows from Proposition \ref{NWessential}.
\end{remark}
Taking into account equation \eqref{eq:gapsin}, it is easy to see that our rate is an improvement of the ones proved in  \cite{lacoste2016convergence} and \cite{bomze2020active} (see Table \ref{tab:5}).
Furthermore, we do not need to start from a vertex to avoid dependence from the support of $\{x_0\}$ like in \cite[Theorem 5.1]{bomze2020active}. Finally, our method improves the conditional gradient sliding rate (NCGS) not only in LMO but also in gradients, given that from $\OM - \{y\} \subset T_{\OM}(y)$ it follows $\pi(y) \leq \n{\pi (T_{\OM}(y), -\nabla f(y))}/2L$ for every $y \in \OM$.

\subsubsection{Objectives with KL property}
As a consequence of Proposition \ref{keyprop}, we have linear convergence rates for the general algorithmic scheme reported in Algorithm \ref{alg:3}  under the KL inequality \eqref{hp:klP}, the angle condition \eqref{eq:ang_cond},  and finite termination of the SSC procedure.
%
In the next results (Lemma \ref{p:crkl}, Theorem \ref{t:loc} and Corollary \ref{cor:glob}), we always assume the following:
	\begin{itemize}
		\item the angle condition \eqref{eq:ang_cond} holds;
		\item the SSC procedure always terminates in a finite number of steps.
\end{itemize}
\begin{lemma}\label{p:crkl}
	Let us  consider the sequence $\{x_k\}$ generated by  Algorithm \ref{alg:3} and assume that 
	the objective function $f$ satisfies condition \eqref{hp:klP}, with $f(x^*)$ fixed, in every feasible point generated by the algorithm.
	Then, for $q = \left(1 + \frac{\mu}{L}\frac{\tau^2}{(1 + \tau)^2}\right)^{-1}$ we have $f(x_k) \rightarrow f(x^*)$, with
	\begin{equation} \label{th:qdec}
		f(x_k) - f(x^*) \leq q^k (f(x_0) - f(x^*))	\, ,
	\end{equation}
	and $x_k \rightarrow \tilde{x}^*$ with
	\begin{equation} \label{th:taillen}
		\n{x_k - \tilde{x}^*} \leq \frac{\sqrt{2-2q} (f(x_0) - f(\tilde x^*))}{\sqrt{L}(1-\sqrt{q})} q^{\frac{k}{2}} \, ,
	\end{equation}
	for $\tilde{x}^*$ stationary point such that $f(\tilde{x}^*) = f(x^*)$.
\end{lemma}
\newcommand{\tdelta}{\tilde{\delta}}
As an example, the assumption of Lemma \ref{p:crkl} is clearly satisfied if \eqref{hp:klP} holds globally, corresponding to a constrained version of the global PL property used in \cite{karimi2016linear}. By \cite[Corollary 6]{bolte2017error}, for convex objectives this assumption is satisfied in particular under a global quadratic Holderian error bound, thus, e.g., by strongly convex objectives. \\
 Under mild assumptions on the stationary point $x^*$, we can also apply Lemma~\ref{p:crkl} locally on non convex objectives, thus adapting to our projection free setting the local results given in \cite[Section 2.3]{attouch2013convergence} for proximal methods.
	\begin{theorem} \label{t:loc}
		Let Assumption \ref{ass:KL} hold at $x^*$. Further assume that $x_k \in B_{\delta}(x^*) \Rightarrow f(x_{k + 1}) \geq f(x^*)$. Then, for some $\tdelta > 0$, if $x_0 \in B_{\tdelta}(x^*)$ the rates \eqref{th:qdec} and \eqref{th:taillen} hold.
	\end{theorem}
	It is not difficult to see that the assumption $x_k \in B_{\delta}(x^*) \Rightarrow f(x_{k + 1}) \geq f(x^*)$ is true, e.g.,  if $x^*$ is a minimizer on its connected component of the sublevel set $[f \leq f(x_0)]$. \\
	
As a corollary of Theorem \ref{t:loc}, we can apply Lemma \ref{p:crkl} and derive the following asymptotic rates.
	\begin{corollary} \label{cor:glob}
		Let us  consider the sequence $\{x_k\}$ generated by  Algorithm \ref{alg:3}. Let Assumption \ref{ass:KL} hold at every point of the limit set of $\{x_k\}$. Then, for some positive constants $M$ and $\bar{M}$, $\{x_k\} \rightarrow x^*$, with the asymptotic rates:
		\begin{equation}
			\begin{aligned}
				& f(x_k) - f(x^*) \leq Mq^k \, , \\
				& \n{x_k - x^*} \leq \bar{M}q^{\frac{k}{2}} \, .
			\end{aligned}
		\end{equation}
	\end{corollary}


\setcounter{table}{0} \renewcommand{\thetable}{\arabic{table}}
\captionsetup[table]{name=Table, labelfont=bf}
\newcommand{\kgd}{N_{g}}	
\newcommand{\qw}{q_w}

\begin{table} 
	\centering
	\bgroup
	\def\arraystretch{2.1}
	\scalebox{0.6}{\begin{tabular}{|l|c|c|c|c|c|c|c|} 
			\hline
			Algorithm & Article & Objective & $\gamma(k)$                   & $I_b$                                                & $\bqfs$      & $h_k/h_0$ upper bound   & $T_{avg}$          \\ \hline
			AFW                           & \cite{lacoste2015global}       & SC        & $k/2$                       & $\vert S_0\vert  - 1$                                             & $1-\frac{\mu}{L}\frac{\tau_p^2}{4}$ & $\left(1-\frac{\mu}{L}\frac{\tau_p^2}{4}\right)^{\frac{k}{2}}$ & - \\ \hline
			PFW                           & \cite{lacoste2015global}       & SC        & $k/(3\vert A\vert ! + 1)$                  & -                                                  & $1-\frac{\mu}{L}\tau_p^2$ & $\left(1-\frac{\mu}{L}\tau_p^2\right)^{\frac{k}{3\vert A\vert ! + 1}}$ & - \\ \hline
			FDFW\footnotemark                       & \cite{kolmogorov2020practical}  & SC        & $k/(\Delta(\OM) + 1)$           & $\tx{dim}(\F(x_0))$               &$1-\frac{\mu}{L}\frac{\tau_v^2}{4}$  &$\left( 1-\frac{\mu}{L}\frac{\tau_v^2}{4}\right)^{\frac{k}{\Delta(\OM) + 1}}$ & -  \\ \hline
			AFW + SSC                         & Ours                            & NC, KL    & $k$                            & -                                                  & $\left( 1 + \frac{\mu}{L}\frac{\tau_p^2}{(2 + \tau_p)^2}\right)^{- 1}$ & $\left( 1 + \frac{\mu}{L}\frac{\tau_p^2}{(2 + \tau_p)^2}\right)^{- k}$ & 2 \\ \hline
			PFW + SSC                         & Ours                            & NC, KL    & $k$                            & -                                                  &$ \left( 1 + \frac{\mu}{L}\frac{\tau_p^2}{(1 + \tau_p)^2}\right)^{- 1} $& $ \left( 1 + \frac{\mu}{L}\frac{\tau_p^2}{(1 + \tau_p)^2}\right)^{- k} $ & 2 \\ \hline
			FDFW + SSC                       & Ours                            & NC, KL    & $k$                           & -                                                  &  $\left( 1 + \frac{\mu}{L}\frac{\tau_v^2}{(1 + \tau_v)^2}\right)^{- 1}$&  $\left( 1 + \frac{\mu}{L}\frac{\tau_v^2}{(1 + \tau_v)^2}\right)^{- k}$ & $\Delta(\OM) + 1$ \\   \hline                                
	\end{tabular}}
	\egroup
	\caption{Comparison between the rates of the standard and SSC version of some FW variants for $\OM = \conv(A)$ with $\vert A\vert  < \infty$. SC = strongly convex, NC = non convex, KL = KL property.  $\gamma(k)$:~lower bound on the number of good steps after $k$ steps, counting from the first good step. $I_b$: bound on the number of bad steps before the first good step. $\bqfs$: rate in good steps. $h_k/h_0 $ upper bound: worst case rate assuming no initial bad steps, equal to $\bqfs^{\gamma(k)}$.    $\Delta(\OM) = $ maximum increase in face dimension $\F(x_{k + 1}) - \F(x_k)$ after a FW step. $S_0 = $ active set for $x_0$. $T_{avg} = $ worst case average iteration number of the SSC (see Proposition \ref{p:SSCworstaverage})}
	\label{tab:2}
\end{table}

Similarly to what we did for Theorem \ref{cor:stationary}, here we give a corollary for Lemma \ref{p:crkl} related to the FW variants described in Section \ref{s:FWv}. 
\begin{corollary} \label{cor:klrate}
	Let us assume that the objective function $f$  satisfies condition~\eqref{hp:klP} on every point generated by the algorithm, with $f(x^*)$ fixed, and that $\Omega = \tx{conv}(A)$ with $\vert A\vert  < +\infty $ in Problem \eqref{eq:mainpb}.
	Then the sequence $\{x_k\}$  generated by  Algorithm \ref{alg:3} with AFW (PFW or  FDFW) in the SSC
	converges at the rates given by Lemma \ref{p:crkl}, with $\tau = \tau_p/2$ ($\tau_p$ or $\tau_v/2$, respectively).
\end{corollary}

\begin{proof}
	Finite termination of the SSC follows by Proposition \ref{p:SSCworst}, and the angle condition is satisfied by Proposition \ref{eqphi}. Thus we have all the assumptions to apply Lemma \ref{p:crkl}.
\end{proof}

For comparison, we now recall some well-known result related to	global linear convergence rates for the FW variants under analysis. 


\begin{proposition} \label{p:sc.wcr}
	Let us assume that the objective function $f$ is $\mu - $strongly convex and $\Omega = \tx{conv}(A)$ with $\vert A\vert  < +\infty $ in Problem \eqref{eq:mainpb}. Let $\{x_k\}$ be a sequence generated by the AFW (PFW or FDFW), with stepsize given by exact linesearch. If the initial active set is $S_0 = \{x_0\}$ for the AFW ($S_0 = \{x_0\}$ for the PFW, $\dim(\F(x_0)) = 0$ for the FDFW), then 
	\begin{equation}
		f(x_k) - f^* \leq \bqfs^{\gamma(k)}(f(x_0) - f^*) \, , 
	\end{equation}
	for $\gamma(k)$ and $\bqfs$ given in Table \ref{tab:2}. 
\end{proposition}
\begin{proof}
	For the AFW and the PFW the result follows directly from \cite[Theorem 1]{lacoste2015global}, with the exception of the good steps rate for the PFW, which can be obtained by applying the bound \cite[Equation 10]{lacoste2015global} in \cite[Equation 5]{lacoste2015global}. For the FDFW the result follows from \cite[Theorem 1]{kolmogorov2020practical} (where the method is referred to as DiCG), with the bound $\mu \PWidth(V(\OM)^2$ on the geometric strong convexity constant implied by \cite[Theorem 6]{lacoste2015global} improved to $\mu \PFWidth(\OM)^2$ as in Proposition \ref{eqphi}.  
\end{proof}

For all the examples where an upper bound on $\tau_p = \frac{\tx{PWidth}(A)}{D}$ is known (see  \cite{rademacher2020smoothed}, \cite{pena2018polytope} and references therein) when $\tx{dim}(\tx{conv}(A)) \rightarrow \infty$ then $\tau_p \rightarrow 0$ and our rates for the SSC converge to the rates without SSC for good steps in Table~\ref{tab:2}. While we are not able to prove this limit in general,  for all polytopes with dimension greater or equal to 2, except low dimensional simplices (see Example~\ref{ex:1}), we still have $\tau_p \leq \frac{1}{2}$ (because $\tx{PdirW}(A, g, x) + \tx{PdirW}(A, -g, x) \leq D$ for $x$ in the relative interior of $\conv(A)$ and $\pm g$ feasible and orthogonal to $\conv(S)$ for some $S \in S_x$). Using this together with Example~ \ref{ex:1} for simplices, it is easy to check that the rates in Corollary \ref{cor:klrate} (SSC based FW variants) are strict improvements on the known worst case rates (standard FW variants) reported in Proposition~\ref{p:sc.wcr}, with a limited number of exceptions. These are the trivial one dimensional case and simplices with low dimension ($\leq 4$ for the PFW, and $\leq 8$ for the AFW using the loose bounds in Example \ref{ex:1}) combined with objectives having condition number $\mu/ L $ sufficiently close to 1.

\begin{example} \label{ex:1}
	If $W(\conv(A))$ is the width of $\conv(A)$ (see \cite[Section~3]{lacoste2015global}) then it follows directly from the definition of $\PWidth$ that $W(\conv(A)) \geq \PWidth(A) $, with equality for $A = \{e_1, ... , e_n\}$ (see \cite{lacoste2015global} and \cite{pena2018polytope}). Let now $A = \{a_1, ..., a_n\}$ be a set of $n$ affinely independent points in $\RR^{n - 1}$. We claim that, for $r_n = \sqrt{1 - \frac{1}{n}}$ circumradius of the $n-1$ dimensional unit simplex $\Delta_{n-1}$ 
	\begin{equation}
		\PWidth(A)/D \leq r_n^{-1}W(\Delta_{n-1})  = \begin{cases}
			2r_n^{-1}\sqrt{\frac{1}{n}} &\tx{ for } n \tx{ even}, \\
			2r_n^{-1}\sqrt{\frac{1}{n - 1/n}}  &\tx{ for } n \tx{ odd}.
		\end{cases}
	\end{equation}
	To see this, assume without loss of generality $D = 1$ and $0 \in \tx{int}(\OM)$ for $ \OM = \conv(A)$. Then if 
	$\hat{A} = \{\hat{a}_1, ..., \hat{a}_n\}$ we have $W(\conv(\hat{A})) \geq W(\conv(A))$. We can conclude
	\begin{equation}
		\frac{\PWidth(A)}{D} = \PWidth(A) \leq W(\conv(A)) \leq W(\conv(\hat{A})) \leq r_n^{-1} W(\Delta_{n-1})  \, ,
	\end{equation} 
	where in the last inequality we used that regular simplices maximize the width among simplices with fixed inradius  (see, e.g., \cite{alexander1977width} and \cite{gritzmann1989estimates}).
\end{example}

\section{Examples} \label{s:examples}

We now discuss some examples of objectives satisfying the KL property and sets where the angle condition can be satisfied with an explicit bound, relevant to practical optimization problems. 
\subsection{KL property}
The KL property of Assumption \ref{ass:KL} is satisfied for Problem \eqref{eq:mainpb} in the following cases:
\begin{itemize}
	\item $f$ is composite strongly convex, i.e. $f(x) = g(Bx)$ with $g$ strongly convex, and $\OM$ is a polytope \cite[Proposition 4.1]{li2018calculus},
	\item $f$ is composite strongly convex as in the previous point, $\OM$ is the $l^p$ ball for $p \in [1, 2]$, and $\inf_{x \in \OM} f(x) > \inf_{x \in \RR^n} g(Bx)$ \cite[Proposition 4.2]{li2018calculus},
	\item $f$ is (non convex) quadratic, i.e. $f(x) = x^{\top}Qx + b^{\top}x + c$, and $\OM$ is a polytope,  \cite[Corollary 5.2]{li2018calculus},
	\item $f$ is non convex quadratic and does not satisfy the degeneracy condition of \cite[equation (30)]{jiang2022holderian}, and $\OM$ is the unit sphere \cite[Theorem 3.13]{jiang2022holderian}.
\end{itemize}
\subsection{Angle condition bounds}
\subsubsection{Bounds using $\PWidth$}
For the unit simplex and the unit cube explicit $\Theta(1/\sqrt{n})$ values were given in \cite[Example 1 and 2]{pena2018polytope}. With analogous arguments it can be proved that the $\PWidth$ of the $l_1$ ball is $1/\sqrt{n}$. By Proposition \eqref{eqphi}, this implies that the angle condition can be lower bounded with $\tau = \Theta(1/\sqrt{n})$ for the unit simplex and the $l_1$ ball, and with $\tau = \Theta(1/n)$ for the unit cube.  
\subsubsection{Bounds using facial distance $\vf$}
For a polytope $\OM = \{ x \in \RR^n \ \vert \ Ax \leq b \}$ with $A \in \RR^{m \times n}$ the facial distance can be defined as (see \cite{beck2017linearly}):
\begin{equation} \label{eq:vfdef}
	\vf(\OM) = \min_{\substack{v \in V(\OM) \\ i:\Sc{a^{(i)}}{v} < b_i}  } \frac{b_i - \Sc{a^{(i)}}{v}}{\n{a^{(i)}}} \, .
\end{equation} 
It is the easy to bound $\vf(\OM)$ on some specific class of polytopes and, consequently, give an explicit bound for the angle condition (see also \cite{bashiri2017decomposition}). For instance, if the matrix $A$ is totally unimodular (i. e. all the vertices are integral for $b$ integral), we have the following properties.
\begin{proposition} \label{p:unimodular}
	 If the matrix $A$ is totally unimodular and $b$ is integral, then for $\bar{a} = \max_{i \in [1:m]} \n{a_i}$:
	 \begin{itemize}
	 	\item for the AFW or the PFW, if the size of the active set stays bounded by $\bar{s}$, then 
	 	\begin{equation}
	 		\SB_{\tx{AFW}}(\OM) \geq \frac{1}{2\bar{s}\bar{a}D}, \quad \SB_{\tx{PFW}}(\OM) \geq \frac{1}{\bar{s}\bar{a}D};
	 	\end{equation}
 	   \item for the FDFW, 
 	   \begin{equation} \label{eq:vfFD}
 	   	\SB_{\tx{FD}}(\OM) \geq \frac{1}{2D\bar{a}(\dim(\OM) + 1)} \geq \frac{1}{2D\bar{a}(n + 1)}.
 	   \end{equation}
	 \end{itemize} 
\end{proposition}
\begin{proof}
	If $A$ is totally unimodular then for $i \in [1:m], v \in V$ such that $b_i - \Sc{a^{(i)}}{v} > 0$ we have
\begin{equation} \label{eq:unitriv}
\frac{b_i - \Sc{a^{(i)}}{v}}{\n{a_i}} \geq \frac{1}{\n{a_i}}
\end{equation}
since the numerator on the LHS must be at least one. By applying \eqref{eq:unitriv} to the RHS of \eqref{eq:vfdef} we obtain
\begin{equation}
	\vf(\OM) \geq \min_{i \in [1:m]} \frac{1}{\n{a_i}} = \frac{1}{\bar{a}} \, .
\end{equation}
Then the thesis follows for the AFW and the PFW directly from the bounds of Remark \ref{r:1}. For the FDFW, the second part of \eqref{eq:vfFD} is trivially true since $\dim(\OM) \leq n$, and the first follows by the bound given in Remark \ref{r:1}, using that by the Caratheodory theorem for every feasible point $x$ there exists $S \in S_x$ with $\vert S \vert \leq \dim(\OM) + 1$.
\end{proof}
The bound of Proposition \ref{p:unimodular} allows us to bound the angle condition for the min cost flow polytope with integral capacities:
\begin{equation} \label{eq:mincost}
	\OM = \{x \in \RR^n \ \vert \ Ax \leq b, \ 0 \leq x \leq c \} \, ,
\end{equation}
with $b, c$ integral and $A$ incidence matrix of a directed graph $G$.
\begin{corollary}
Consider a directed graph $G$ with incidence matrix $A \in \RR^{m \times n}$ and maximum degree of a vertex $d$. Then if $\OM$ is given as in \eqref{eq:mincost}:
\begin{equation}
	\SB_{\tx{FD}}(\OM) \geq \frac{1}{2\sqrt{d}(n + 1)\n{c}}
\end{equation}
\end{corollary}
\begin{proof}
	By the capacity constraints, the diameter of $\OM$ is at most $\n{c}$. Then
		the result follows easily from Proposition \ref{p:unimodular} by noticing that $\OM$ can be rewritten as $\{x \in \RR^n \ \vert \ \bar{A}x \leq b\}$ for $\bar{A} = (A; I; -I)$ totally unimodular (see, e.g., \cite{truemper1977unimodular}) with maximum norm of a row equal to $\sqrt{d}$. 
\end{proof}
\subsubsection{Bounds on sets with smooth boundary}
On convex sets with smooth boundary the angle condition can be satisfied with constant arbitrarily close to 1 using orthographic retractions \cite[Section 6.3]{rinaldi2020unifying}. Furthermore, on sublevel sets of smooth and strongly convex functions the FDFW satisfies the angle condition with constant equal to the condition number of the function divided by 2 \cite[Section 6.2]{rinaldi2020unifying}.

\subsection{Applications}
There is a number of practical optimization problems with the feasible sets and objectives discussed above. To start with, the LASSO problem, the minimum enclosing ball problem, training linear support vector machines and finding maximal cliques in graphs can all be formulated as convex quadratic optimization problems \cite{bomze2021frank} on the $l_1$ ball or the simplex. The trust region subproblem is a non convex quadratic problem on the unit sphere (see \cite{jiang2022holderian}). The min cost flow problem with a quadratic objective is also of practical interest \cite{tamir1993strongly}. Many other examples can be found in \cite{li2018calculus}.

\section{Numerical tests} \label{REST}
We tested the SSC on the AFW and the PFW methods, applied to a quadratic (non convex) relaxation of the maximum clique problem proposed in \cite{bomze1997evolution}. \\ 
More precisely, let $A$ be the adjacency matrix of a graph $G$. In \cite{bomze1997evolution} it is proved that there is a one to one correspondence between the maximal cliques of $G$ and the local minima of the function $f:\Delta_{n-1} \rightarrow \RR$ defined by
\begin{equation} \label{eq:mcobj}
	f(x) = - x^\intercal A x - \frac{1}{2} \|x\|^2.
\end{equation}
Therefore, we consider instances of Problem \eqref{eq:mainpb} with objective \eqref{eq:mcobj} and feasible set the $n-1$ dimensional unit simplex, that is $\OM = \Delta_{n-1}$. \\
The graph instances we use are taken from the DIMACS benchmark \cite{johnson1993cliques}. 
To have a fair comparison for both the AFW and the PFW we use the stepsize given by 
\begin{equation}
	\alpha_k = \min \{\alpha^{\max}_k, -\frac{\Sc{\nabla f(x_k)}{d_k}}{L \n{d_k}^2} \}
\end{equation}
with $\alpha_k^{\max}$ determined by boundary conditions. In this way the new point computed by the methods coincides with the first point computed in the SSC procedure of their multistep versions. \\ 
We reported in Table \ref{tab:7}, \ref{tab:8} the results for the most challenging instances, aggregated on 100 runs starting from random points. The SSC clearly improves the CPU times while keeping the solution quality. Indeed in these problems the SSC allows the methods to identify the support of a local minimum in fewer iterations, so that the slow initial convergence phase is skipped (see Figures \ref{fig:image3}, \ref{fig:image2}).  
\setcounter{table}{2}
\captionsetup[table]{name=Table, labelfont=bf}
\begin{table} 
	\caption{Max clique found, average clique size, standard deviation of clique sizes and average CPU time for AFW and SSC + AFW on max clique instances from the DIMACS benchmark.}
	\begin{tabular}{|l|llll|llll|} \hline
		\multicolumn{1}{|c}{}         & \multicolumn{4}{|c|}{\textbf{AFW}}                                                                                     & \multicolumn{4}{|c|}{\textbf{SSC + AFW}}                                                                                 \\ \hline
		\multicolumn{1}{|c}{Instance} & \multicolumn{1}{|c}{Max} & \multicolumn{1}{c}{Mean} & \multicolumn{1}{c}{Std} & \multicolumn{1}{c}{CPU time} & \multicolumn{1}{|c}{Max} & \multicolumn{1}{c}{Mean} & \multicolumn{1}{c}{Std} & \multicolumn{1}{c|}{CPU time} \\ \hline
		C2000.5          & 14   & 11.7   & 0.89 & 2.800    & 14   & 11.6   & 1.00 & 0.082    \\
		C2000.9          & 67   & 60.2   & 2.20 & 3.135    & 65   & 60.0   & 2.05 & 0.200    \\
		C4000.5          & 16   & 12.8   & 0.94 & 23.487   & 16   & 12.5   & 0.92 & 0.429    \\
		MANN\_a81        & 1080 & 1080.0 & 0.00 & 31.156   & 1080 & 1080.0 & 0.00 & 25.047   \\
		keller6          & 45   & 38.4   & 2.41 & 13.713   & 43   & 37.8   & 2.22 & 0.413    \\
		
		\hline 
	\end{tabular}
	\label{tab:7}
\end{table}
\begin{table} 
	\caption{Max clique found, average clique size, standard deviation of clique sizes and average CPU time for PFW and SSC + PFW on max clique instances from the DIMACS benchmark.}
	\begin{tabular}{|l|llll|llll|} \hline
		\multicolumn{1}{|c}{}         & \multicolumn{4}{|c|}{\textbf{PFW}}                                                                                     & \multicolumn{4}{|c|}{\textbf{SSC + PFW}}                                                                                 \\ \hline
		\multicolumn{1}{|c}{Instance} & \multicolumn{1}{|c}{Max} & \multicolumn{1}{c}{Mean} & \multicolumn{1}{c}{Std} & \multicolumn{1}{c}{CPU time} & \multicolumn{1}{|c}{Max} & \multicolumn{1}{c}{Mean} & \multicolumn{1}{c}{Std} & \multicolumn{1}{c|}{CPU time} \\ \hline
		C2000.5          & 14   & 11.8  & 0.86 & 2.811    & 14   & 12.1  & 0.86 & 0.077    \\
		C2000.9          & 67   & 62.3  & 1.83 & 3.031    & 68   & 62.0  & 1.77 & 0.150    \\
		C4000.5          & 15   & 12.7  & 0.92 & 23.423   & 16   & 13.4  & 0.95 & 0.379    \\
		MANN\_a81        & 1080 & 1080.0  & 0.00    & 19.867   & 1080 & 1080.0  & 0.00    & 15.442   \\
		keller6          & 44   & 37.3  & 2.68 & 13.515   & 45   & 35.6  & 2.83 & 0.258    \\
		\hline 
	\end{tabular}
	\label{tab:8}
\end{table}
\clearpage

\begin{figure}[h]
	\centering
	
	\begin{subfigure}{0.4\textwidth}
		\includegraphics[width=0.9\linewidth, height=4cm]{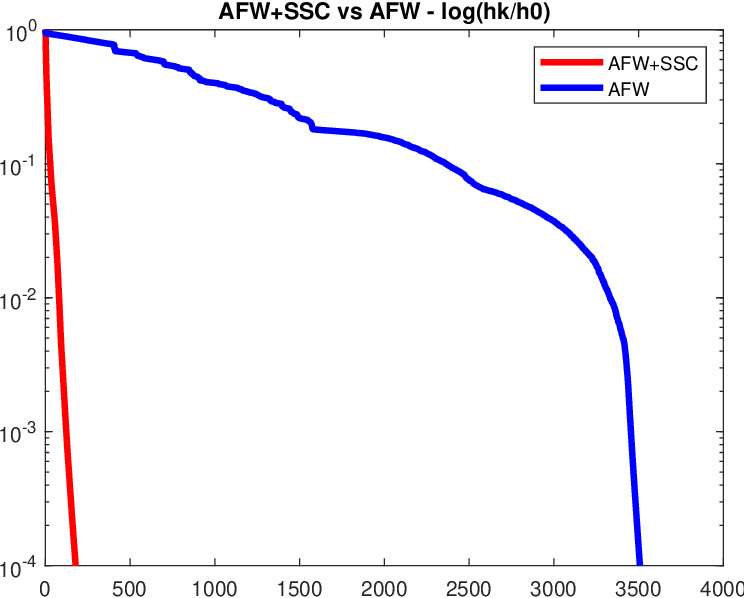} 
	\end{subfigure}
	\begin{subfigure}{0.4\textwidth}
		\includegraphics[width=0.9\linewidth, height=4cm]{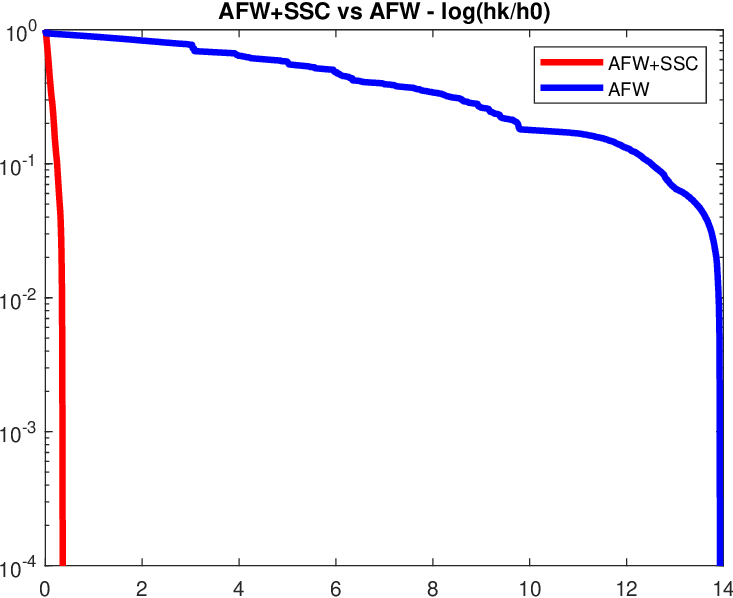}
	\end{subfigure}
	
	\begin{subfigure}{0.4\textwidth}
		\includegraphics[width=0.9\linewidth, height=4cm]{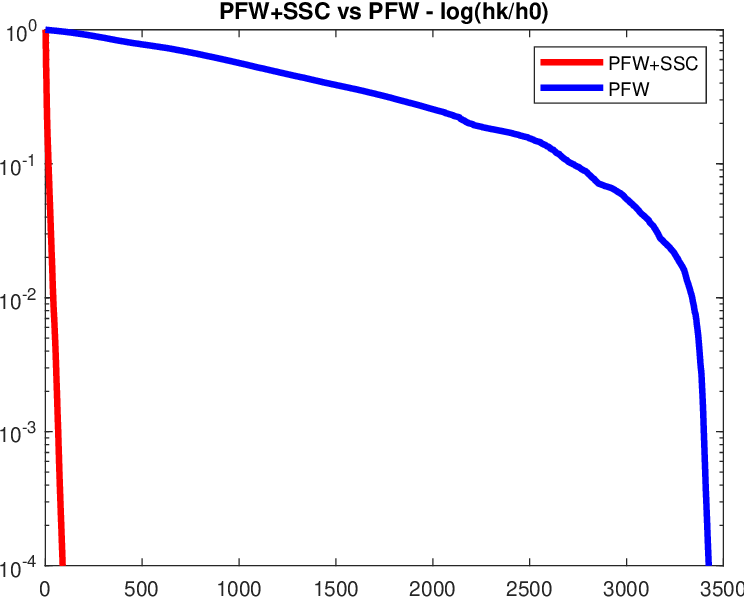} 
	\end{subfigure}
	\begin{subfigure}{0.4\textwidth}
		\includegraphics[width=0.9\linewidth, height=4cm]{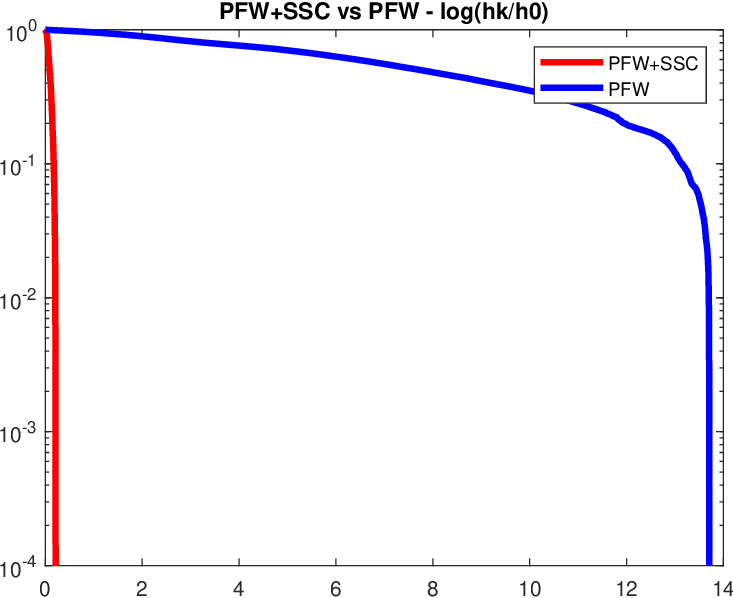}
	\end{subfigure}
	\caption{Iteration number and CPU time vs $\log(h_k/h_0)$ in the first and the second column respectively for the instance keller6}
	\label{fig:image3}
\end{figure}

\begin{figure}[h]
	\centering
		
	\begin{subfigure}{0.4\textwidth}
	\includegraphics[width=0.9\linewidth, height=4cm]{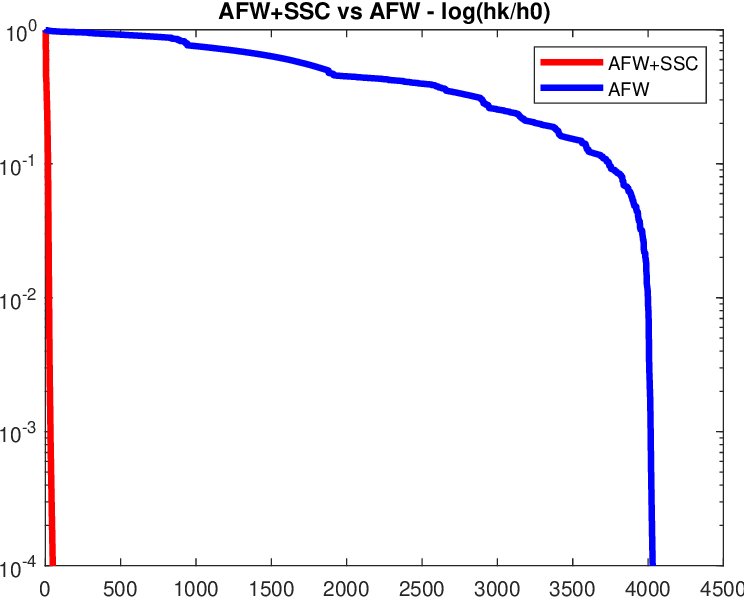} 
\end{subfigure}
\begin{subfigure}{0.4\textwidth}
	\includegraphics[width=0.9\linewidth, height=4cm]{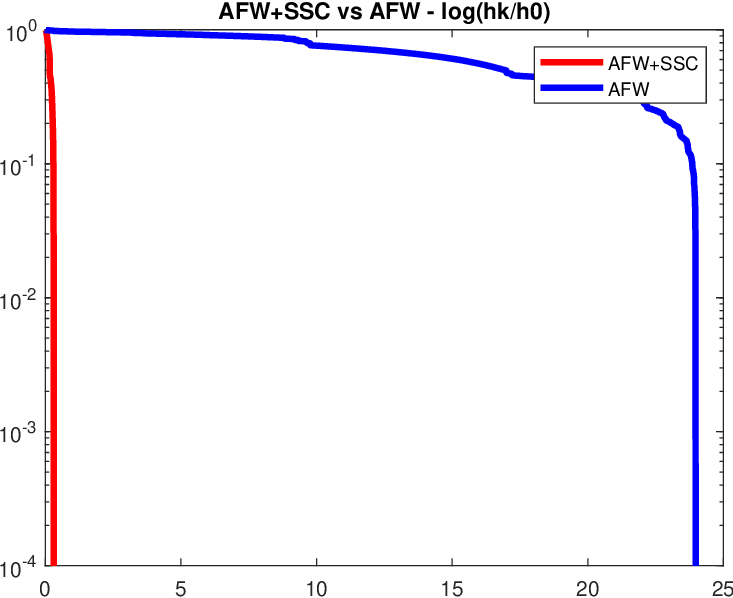}
\end{subfigure}

\begin{subfigure}{0.4\textwidth}
	\includegraphics[width=0.9\linewidth, height=4cm]{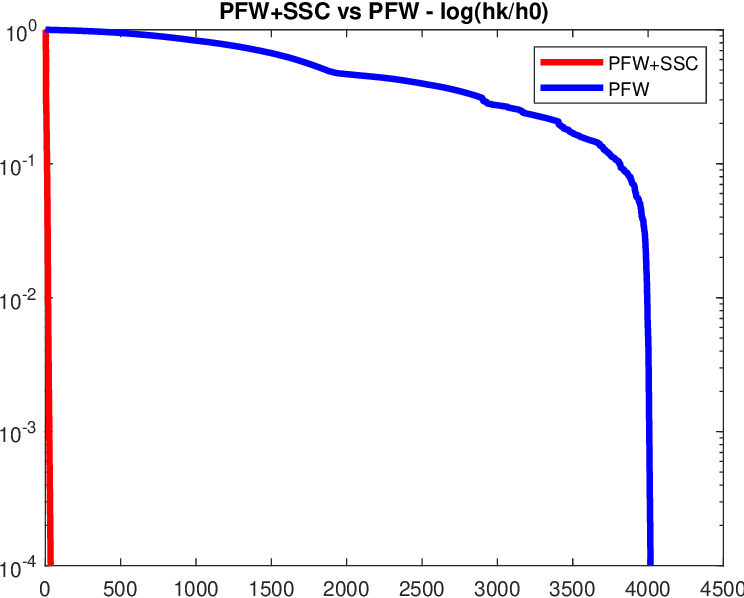} 
\end{subfigure}
\begin{subfigure}{0.4\textwidth}
	\includegraphics[width=0.9\linewidth, height=4cm]{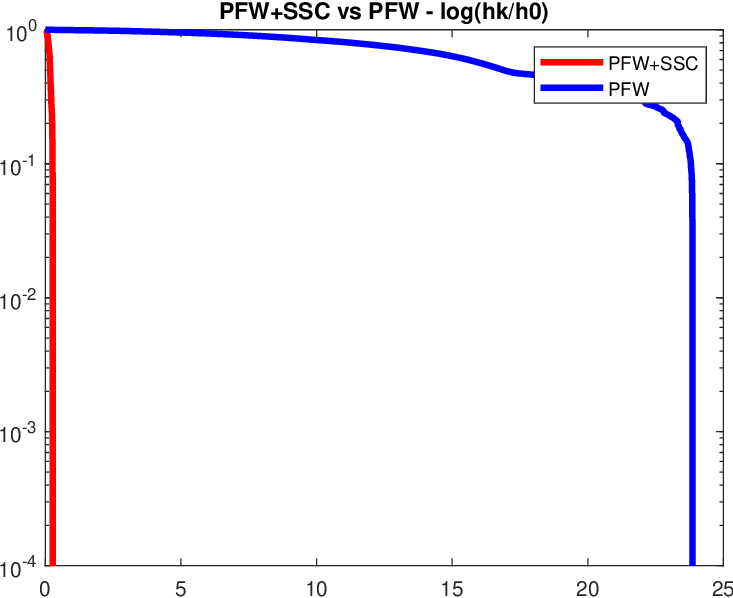}
\end{subfigure}
	\caption{Iteration number and CPU time vs $\log(h_k/h_0)$ in the first and the second column respectively for the instance C4000.5}
	\label{fig:image2}
\end{figure}

\section{Conclusions} 	
FW variants rely on the choice of good feasible descent directions, for which there needs to be a trade-off between slope and maximal stepsize. To address this issue we proposed the SSC procedure, which allowed us to prove bad step free convergence rates under an angle condition for the directions selected by the method. Preliminary numerical experiments also support the soundness of this approach. 

Future research directions include employing our framework to design and analyze other projection free first order methods, investigating active set identification properties of FW variants with the SSC, generalizing our framework to constrained stochastic optimization, as well as applications for the solution of real-world data science problems. 
\section{Appendix}	
\subsection{KL property} \label{s:KLPL}

We state here a result showing an implication between the (global) PL property used in \cite{karimi2016linear} and \eqref{hp:klP}. We first recall the PL property used in \cite{karimi2016linear}:
\begin{equation} \label{eq:PLp}
	\frac{1}{2}\n{\nabla f(x)}^2 \geq \mu(f(x) - f^*) \, .  
\end{equation}
with $f^*$ optimal value of $f$ with non empty solution set $\mathcal{X}^*$. 
\begin{proposition} \label{p:PLiKL}
	If $f$ is convex, the optimal solution set $\mathcal{X}^*$ of $f$ is contained in $\OM$ and \eqref{eq:PLp} holds, then \eqref{hp:klP} holds for every $x \in \OM$.
\end{proposition}
\begin{proof}
	By \cite[Theorem 2]{karimi2016linear} the PL property is equivalent, for convex objectives, to the unconstrained quadratic growth condition:
	\begin{equation} \label{eq:QG}
		f(x) - f^* \geq \frac{\mu}{2}\dist(x, \mathcal{X}^*)^2
	\end{equation}
	In turn, given that by the assumption $\mathcal{X}^* \subset \OM$ the set $\mathcal{X}^*$ is the solution set for $f_{\OM}$ as well, \eqref{eq:QG} implies the global non smooth Holderian error bound condition from \cite{bolte2017error} with $\varphi(t) = \sqrt{\frac{2t}{\mu}} $, and by \cite[Corollary 6]{bolte2017error} this is equivalent to the KL property \eqref{hp:klP} holding globally on $\OM$.
\end{proof}
\begin{remark} \label{r:counterexample}
	We remark that without the assumption $\mathcal{X}^* \subset \OM$ the implication is no longer true even for convex objectives, a counter example being $\OM$ equal to the unitary ball and $f((x^{(1)}, ..., x^{(n)})) = (x^{(1)} - 1)^2$. At the same time, the KL property we used does not imply the PL property in general, since the latter only deals with unconstrained minima. 
\end{remark}

\subsection{Proofs}
We report here the missing proofs. We start with the proof of Lemma \ref{l:suffdec}.
\begin{proof}
			By the standard descent lemma \cite[Proposition 6.1.2]{bertsekas2015convex}, 
	\begin{equation} \label{eq:stdc}
		f(x_{k + 1}) = f(x_k + \alpha_k d_k) \leq f(x_k) + \alpha_k \s{\nabla f(x_k)}{d_k} + \alpha_k^2 \frac{L}{2}\n{d_k}^2 \, ,
	\end{equation}
	and in particular
	\begin{equation} \label{eq:sdec}
		f(x_{k}) - f(x_{k+1}) \geq - \alpha_k \s{\nabla f(x_k)}{d_k} - \alpha_k^2 \frac{L}{2}\n{d_k}^2 \geq \frac{L}{2}\alpha_k^2\n{d_k}^2 = \frac{L}{2}\n{x_{k + 1} - x_k}^2 \, ,
	\end{equation}
	where we used $\alpha_k \leq \bar{\alpha}_k$ in the last inequality. This proves \eqref{eq:suffdec}.
\end{proof}
 We now state a preliminary result needed to prove Proposition~\ref{NWessential}:
\begin{proposition} \label{ddual}
	Let $C$ be a closed convex cone. For every $y \in \mathbb{R}^n$ 
	\begin{equation*}
		\textnormal{dist}(C^*, y) = \sup_{c \in C} \s{\hat{c}}{y} \, .
	\end{equation*}
\end{proposition}
As stated in \cite{burke1988identification} this is an immediate consequence of the Moreau-Yosida decomposition: 
$$ y = \pi(C, y) + \pi(C^*, y) \, . $$ 
\begin{proof}[Proposition \ref{NWessential}.]
	First, by continuity of the scalar product we have
	\begin{equation} \label{omtomeq}
		\sup_{h\in\Omega/\{\bar{x}\}} \left (g, \frac{h-\bar{x}}{\|h - \bar{x}\| }\right ) = \sup_{h\in T_{\Omega}(\bar{x}) \sm \{0\}} (g, \hat{h}) \, .
	\end{equation}
	Since $N_{\Omega}(\bar{x}) = T_{\Omega}(\bar{x})^*$ the first equality is exactly the one of Proposition \ref{ddual} if $g \notin N_{\Omega}(\bar{x})$, and it is trivial since both terms are clearly $0$ if $g \in N_{\Omega}(\bar{x})$.\\
	It remains to prove 
	\begin{equation*}
		\dist(N_{\Omega}(\bar{x}), g)= \|\pi(T_{\Omega}(\bar{x}), g)\| \, ,	
	\end{equation*}
	which is true by the Moreau - Yosida decomposition. 
\end{proof}

\begin{proof}[Proposition \ref{keyprop}.]
	Let $B_j = \BB_{\s{g}{\hat{d}_j}/L}(x_k)$ and let $T$ be such that $x_{k +  1} = y_T$.  \\ 
	Inequality \eqref{eq:dcondy} applied with $j = T$ gives \eqref{dq}. Moreover, by taking $\tilde{x}_k = y_{\tilde{T}}$ for some $\tilde{T} \in [0:T]$ the conditions 
	\begin{equation} \label{xktildecond}
		f(x_{k+1}) \leq f(\tilde{x}_k) \leq f(x_k) - \frac{L}{2} \n{x_k - \tilde{x}_k}^2 
	\end{equation}
	are satisfied by Lemma \ref{l:decreasing} and \eqref{eq:dcondy}. \\
	Let now $p_j = \n{\pi(T_{\OM}(y_j), -\nabla f(y_j))}$ and $\tilde{p}_j = \n{\pi(T_{\OM}(y_j), g)} = \n{\pi(T_{\OM}(y_j), -\nabla f(x_k))}$. We have
	\begin{equation} \label{projineq}
		\begin{aligned}
			\vert p_j - \tilde{p}_j \vert  \leq L\n{y_j-x_k} \, ,
		\end{aligned}		
	\end{equation}
	reasoning as for \eqref{eq:projineq}. We now distinguish four cases according to how the SSC terminates. \\
	\textbf{Case 1:} $T = 0$ or $d_T = 0$. Since there are no descent directions $x_{k +  1} = y_T$ must be stationary for the gradient $g$.  Equivalently, $\tilde{p}_T = \n{\pi(T_{\OM}(x_{k +  1}), g)} = 0$. We can now write
	\begin{equation*}
		\n{x_{k +  1}-x_k} \geq \frac{1}{L}( \vert p_T - \tilde{p}_T \vert ) = \frac{p_T}{L} > \K p_T \, ,
	\end{equation*}
	where we used \eqref{projineq} in the first inequality and $\tilde{p}_T = 0$ in the equality. Finally, it is clear that if $T = 0$ then $d_0 =0$, since $y_0$ must be stationary for $-g$.  \\
	Before examining the remaining cases we remark that if the SSC terminates in Phase II then $\alpha_{T- 1} = \beta_{T-1}$ must be maximal w.r.t. the conditions $y_T \in B_{T-1}$ or $y_T \in \bar{B}$. If $\alpha_{T-1} = 0$ then $y_{T-1} = y_T$, and in this case we cannot have $y_{T-1} \in \partial \bar{B}$, otherwise the SSC would terminate in Phase II of the previous cycle. Therefore necessarily $y_T = y_{T-1} \in \tx{int}(B_{T-1})^c$ (Case 2). If $\beta_{T - 1} = \alpha_{T- 1} > 0$ we must have $y_{T-1}\in \OM_{T-1} = B_{T-1} \cap \bar{B}$, and $y_T \in \partial B_{T - 1}$ (case 3) or $y_T \in \partial \bar{B}$ (case 4) respectively.  \\ 
	\textbf{Case 2:} $y_{T-1} = y_T \in \tx{int}(B_{T-1})^c$. We can rewrite the condition as
	\begin{equation} \label{trueT}
		\s{g}{\hat{d}_{T-1}} \leq L\n{y_{T-1} - x_k} = L \n{y_T - x_k} \, .
	\end{equation}
	Thus 
	\begin{equation} \label{t2T}
		p_T = p_{T-1} \leq \tilde{p}_{T-1} + L\n{y_{T} - x_k} \leq \frac{1}{\tau}\s{g}{\hat{d}_{T-1}} + L\n{y_T - x_k} \leq \left(\frac{L}{\tau} + L\right) \n{y_T - x_k} \, ,
	\end{equation}
	where in the equality we used $y_{T} = y_{T-1}$, the first inequality follows from \eqref{projineq} and again $y_T = y_{T-1}$, the second from $\frac{\s{g}{\hat{d}_T}}{\tilde{p}_T} \geq \DSB_{\mathcal{A}}(\OM, y_{T}, g) \geq \SB_{\mathcal{A}}(\OM) = \tau$, and the third from \eqref{trueT}. Then $\tilde{x}_k = x_{k + 1} = y_T$ satisfies the desired conditions. \\ 		
	\textbf{Case 3:} $y_T = y_{T - 1} + \beta_{T - 1} d_{T-1}$ and $y_T \in \partial B_{T-1}$. Then from $y_{T-1} \in B_{T-1}$ it follows
	\begin{equation} \label{c3y}
		L \n{y_{T-1} - x_k} \leq \s{g}{\hat{d}_{T-1}} \, ,
	\end{equation} 
	and $y_T \in \partial B_{T-1}$ implies
	\begin{equation}\label{t1T}
		\s{g}{\hat{d}_{T-1}} = L \n{y_T - x_k} \, .
	\end{equation}
	Combining \eqref{c3y} with \eqref{t1T} we obtain 
	\begin{equation} \label{yt1}
		L \n{y_{T - 1} - x_k} \leq L \n{y_T - x_k} \, .
	\end{equation}
	Thus 
	\begin{equation*}
		p_{T - 1} \leq \tilde{p}_{T - 1} + L\n{y_{T - 1} - x_k} \leq \frac{1}{\tau}\s{g}{\hat{d}_{T - 1}} + L\n{y_{T - 1} - x_k} \leq \left(\frac{L}{\tau} + L\right) \n{y_T - x_k}	\, ,	
	\end{equation*}
	where we used \eqref{t1T}, \eqref{yt1} in the last inequality and the rest follows reasoning as for \eqref{t2T}. In particular we can take $\tilde{x}_k = y_{T-1}$, where $\n{\tilde{x}_k - x_k} \leq \n{x_{k + 1} - x_k}$ by \eqref{yt1}. \\
	\textbf{Case 4:} $y_T = y_{T - 1} + \beta_{T - 1} d_{T-1}$ and $y_T \in \partial \bar{B}$. \\
	The condition $x_{k +  1} = y_T \in \bar{B}$ can be rewritten as
	\begin{equation} \label{case2c}
		L\n{x_{k +  1} - x_k}^2 - \s{g}{x_{k +  1} - x_k} = 0 \, .
	\end{equation}
	For every $j \in [0:T]$ we have
	\begin{equation} \label{eq:rec}
		x_{k +  1} = y_j + \sum_{i=j}^{T-1} \alpha_i d_i \, .
	\end{equation}
	We now want to prove that for every $j \in [0:T]$ 
	\begin{equation} \label{eq:claim2}
		\n{ x_{k +  1} - x_k} \geq \n{y_j - x_k} \, .
	\end{equation}
	Indeed, we have
	\begin{equation*}
		\begin{aligned}
			L\n{ x_{k +  1} - x_k}^2 = \s{g}{x_{k +  1} - x_k} &= \s{g}{y_j - x_k} + \sum_{i=j}^{T-1} \alpha_i \s{g}{d_i} \\ &\geq \s{g}{y_j - x_k} \geq L\n{y_j - x_k}^2 \, ,
		\end{aligned}
	\end{equation*}
	where we used \eqref{case2c} in the first equality, \eqref{eq:rec} in the second, $\s{g}{d_j} \geq 0$ for every $j$ in the first inequality and $y_j \in \bar{B}$ in the second inequality. \\
	We also have 
	\begin{equation} \label{eq:fracineq}
		\begin{aligned}
			\frac{\s{g}{x_{k +  1} - x_k}}{\n{x_{k +  1} - x_k}}  = \frac{\s{g}{\sum_{j=0}^{T-1}\alpha_j d_j}}{\n{\sum_{j=0}^{T-1}\alpha_j d_j}} & \geq
			\frac{\s{g}{\sum_{j=0}^{T-1}\alpha_j d_j}}{\sum_{j=0}^{T-1}\alpha_j \n{d_j}} \\
			\geq \min \left\{\frac{\s{g}{d_j}}{\n{d_j}} \ \vert \ 0  \leq j \leq T-1 \right\} \, .
		\end{aligned}
	\end{equation}
	Thus for $\tilde{T} \in \argmin \left\{ \frac{\s{g}{d_j}}{\n{d_j}} \ \vert \ 0  \leq j \leq T-1 \right\}$ 
	\begin{equation} \label{tildeT}
		\s{g}{\hat{d}_{\tilde{T}}} \leq 	\frac{\s{g}{x_{k +  1} - x_k}}{\n{x_{k +  1} - x_k}} = L\n{x_{k +  1} - x_k} \, ,
	\end{equation}
	where we used \eqref{eq:fracineq} in the first inequality and \eqref{case2c} in the second. \\
	We finally have
	\begin{equation*}
		p_{\tilde{T}} \leq \tilde{p}_{\tilde{T}} + L\n{y_{\tilde{T}} - x_k} \leq \frac{1}{\tau}\s{g}{\hat{d}_{\tilde{T}}} + L\n{y_{ \tilde{T}} - x_k} \leq \left(\frac{L}{\tau} + L\right) \n{x_{k +  1} - x_k} \, ,
	\end{equation*}
	where we used \eqref{eq:claim2}, \eqref{tildeT} in the last inequality and the rest follows reasoning as for \eqref{t2T}. In particular $\tilde{x}_k = y_{\tilde{T}}$ satisfies the desired properties, where $\n{\tilde{x}_k - x_k} \leq \n{x_{k + 1} - x_k}$ by \eqref{eq:claim2}. 
\end{proof}

	\begin{proof}[Proof of Proposition \ref{p:SSCworstaverage}.]
	Let $T(k)$ be the number of iterates generated by the SSC at the step $k$ in Phase II.
	For the AFW and the PFW, reasoning as in the proof of Proposition \ref{p:SSCworst} we obtain that if the SSC does $T(k)$ iterations, the number of active vertices decreases by at least $T(k) - 2$. Then on the one hand
	\begin{equation} \label{eq:afwFxkbound}
		\vert S^{(k)} \vert - \vert S^{(0)} \vert \geq 1 - \vert S^{(0)} \vert \, ,
	\end{equation}
	while on the other hand
	\begin{equation} \label{eq:afwFxkbound2}
		\begin{aligned}
			& \vert S^{(k)} \vert - \vert S^{(0)} \vert = \sum_{i = 0}^{k - 1} (\vert S^{(i + 1)} \vert - \vert S^{(i)} \vert) \\ 
			& \leq 2k - \sum_{i = 0}^{k - 1} T(i) \, .	
		\end{aligned}
	\end{equation}
	Combining \eqref{eq:afwFxkbound} and \eqref{eq:afwFxkbound2} and rearranging, we obtain:
	\begin{equation}
		\frac{1}{k}\sum_{i = 0}^{k - 1} T(i) \leq 2 + \frac{\vert S^{(0)} \vert - 1}{k} \, ,
	\end{equation}
	and the desired result follows by taking the limit for $k \rightarrow \infty$.\\
	For the FDFW, notice that at every iteration the SSC performs a sequence of maximal in face steps terminated either by a Frank Wolfe step, after which $\F(y_j)$ can increase of at most $\Delta(\OM)$, or by a non maximal in face step, after which $\F(y_j)$ stays the same. In both cases, we have
	\begin{equation}
		\dim (\F(x_{ k + 1})) - \dim(\F(x_k)) \leq \Delta(\OM) - T(k) + 1.
	\end{equation}
	Then,
	\begin{equation} \label{eq:Fxkbound}
		\dim \F(x_{k}) - \dim \F(x_0) \geq - \dim \F(x_0) \, ,
	\end{equation}
	and
	\begin{equation} \label{eq:Fxkbound2}
		\begin{aligned}
			& \dim \F(x_k) - \dim \F(x_0) = \sum_{i = 0}^{k - 1} (\dim(\F(x_{i + 1}) - \dim(\F(x_i)))) \\ 
			& \leq k\Delta(\OM) + k - \sum_{i = 0}^{k - 1} T(i) \, .	
		\end{aligned}
	\end{equation}
	The conclusion follows as for the AFW and the PFW. 
\end{proof}

\begin{proof}[Theorem \ref{cor:stationary}.]
	The sequence $\{f(x_k)\}$ is decreasing by \eqref{dq}. Thus by compactness $f(x_k) \rightarrow \tilde{f} \in \RR$ and in particular $f(x_k) - f(x_{k + 1}) \rightarrow 0$. So that by \eqref{dq} also $\n{x_{k + 1} - x_k }\rightarrow 0$.  Let $\{x_{k(i)}\} \rightarrow \tilde{x}^*$ be any convergent subsequence of $\{x_k\}$. For $\{\tilde{x}_k\}$ chosen as in the proof of Proposition \ref{keyprop} we have $\n{\tilde{x}_k - x_k} \leq \n{x_{k+1} - x_k}$ because $\tilde{x}_k = y_T = x_k$ in case 1 and case 2, by \eqref{yt1} in case 3, and by \eqref{eq:claim2} in case 4. Therefore
	$$\n{\tilde{x}_{k(i)} - x_{k(i)}} \leq \n{x_{k(i)+ 1} - x_{k(i) }} \rightarrow 0 \, .$$ Furthermore,  $\n{\pi(T_{\OM}(\tilde{x}_{k(i)}), -\nabla f(\tilde{x}_{k(i)})))}~\leq~\frac{\n{x_{k(i)+ 1} - x_{k(i) }}}{K} \rightarrow 0 $ again by Proposition \ref{keyprop}, so that $\tilde{x}_{k(i)} \rightarrow \tilde{x}^*$ with $\n{\pi(T_{\OM}(\tilde{x}_{k(i)}), -\nabla f(\tilde{x}_{k(i)}))} \rightarrow 0$. Then $\n{\pi(T_{\OM}(\tilde{x}^*), -\nabla f(\tilde{x}^*))} =0 $ and $\tilde{x}^*$ is stationary. 
	
	The first inequality in \eqref{eq:qsqrk} follows directly from \eqref{dq2}. As for the second, we have 
	\begin{equation*} 
		\begin{aligned}
			& \frac{k + 1}{\K ^2} (\min_{0 \leq i \leq k} \n{x_{i + 1} - x_i})^2 = \frac{k + 1}{\K ^2} \min_{0 \leq i \leq k} \n{x_{i + 1} - x_i}^2  \\
			\leq &\frac{1}{\K ^2} \sum_{i= 0}^k \n{x_{i} - x_{i + 1}}^2 \leq \frac{2}{L\K ^2} \sum_{i = 0}^{k}(f(x_{i + 1}) - f(x_i)) \leq \frac{2(f(x_0) - \tilde{f})}{L\K ^2} \, ,
		\end{aligned}
	\end{equation*}
	where we used \eqref{dq} in the first inequality, $\{f(x_i)\}$ decreasing together with $f(x_i) \rightarrow \tilde{f}$  in the second and the thesis follows by rearranging terms. 
\end{proof}

We now prove Lemma \ref{p:crkl}. We start by recalling Karamata's inequality (\cite{kadelburg2005inequalities}, \cite{karamata1932inegalite}) for concave functions. Given $A, B \in \RR^N$ it is said that $A$ majorizes $B$, written $A \succ B$, if
\begin{equation*}
	\begin{aligned}
		\sum_{i= 1}^j A_i & \geq \sum_{i= 1}^j B_i \ \tx{for } j \in [1:N] \, , \\
		\sum_{i = 1}^N A_i & = \sum_{i= 1}^N B_i \, .
	\end{aligned}		
\end{equation*}
If $h$ is concave and $A \succ B$ by Karamata's inequality
\begin{equation*} 
	\sum_{i= 1}^N h(A_i) \leq \sum_{i = 1}^N h(B_i) \, .
\end{equation*}

\newcommand{\tf}{\tilde{f}}

In order to prove Lemma \ref{p:crkl} we first need the following technical Lemma. 
	\begin{lemma}\label{l:karamata}
		Let $\{\tf_i\}_{i \in [0:j]}$ be a sequence of nonnegative numbers such that $\tf_{i + 1} \leq q \tf_i$ for some $q < 1$. Then 
		\begin{equation}
			\sum_{i = 0}^{j - 1} \sqrt{\tf_i - \tf_{i + 1}} \leq  \frac{\sqrt{\tf_0(1 - q)}}{1 - \sqrt{q}} \, .
		\end{equation} 
	\end{lemma}
	\begin{proof}
		Let $\bar{j} = \max\{i \geq 0 \ \vert \ \tf_j \leq q^{i} \tf_0 \}$, so that by \eqref{eq:qdec} we have $\bar{j} \geq j$. Define $w^*, v \in \mathbb{R}^{\bar{j} + 1}_{\geq 0}$ by
		\begin{equation}
			\begin{aligned}
				v &= (\tf_0 - q \tf_0, ..., q^{\bar{j} - 1}\tf_0 - q^{\bar{j}} \tf_0, q^{\bar{j}} \tf_0 - \tf_{j}) \, , \\
				w^* &= (\tf_0 - \tf_1, ..., \tf_{j - 1} - \tf_{j}, 0, ..., 0) \, .
			\end{aligned}
		\end{equation}
		Then for $0 \leq l < \bar{j}$ we have
		\begin{equation} \label{eq:vk1}
			\sum_{i = 0}^l v_i = \tf_0 - q^{l + 1} \tf_0 \leq \tf_{0} - \tf_{\min(l+1, j)} = \sum_{i = 0}^l w^*_i \, ,
		\end{equation}
		where we used $q^{l + 1}\tf_0 \geq \tf_{l + 1} $ for $l \leq j - 1$ and $q^{l + 1}\tf_0 \geq \tf_{j}$ for $j \leq l < \bar{j}$ in the inequality. Furthermore, for $l = \bar{j}$ we have
	\begin{equation} \label{eq:vk3}
		\sum_{i= 0}^l v_i = \tf_{0} - \tf_{j} = \sum_{i=0}^{l} w^*_i \, .
	\end{equation}
		Now if $w$ is the permutation in descreasing order of $w^*$, clearly thanks to \eqref{eq:vk1}, and \eqref{eq:vk3} we have $w \succ v$.	Then 
		\begin{equation}
			\begin{aligned}
			\sum_{i = 0}^{j - 1} \sqrt{\tf_i - \tf_{i +1}} = & \sum_{i= 0}^{\bar{j} + 1} \sqrt{w^*_i} = \sum_{i= 0}^{\bar{j} + 1} \sqrt{w_i} \leq  \sum_{i= 0}^{\bar{j} + 1} \sqrt{v_i} \\ 
				\leq &  \sqrt{\tf_0} \sum_{i= 0}^{+ \infty}\sqrt{q^i - q^{i + 1}} = \frac{\sqrt{\tf_0(1 - q)}}{1 - \sqrt{q}}	\, ,
			\end{aligned}
		\end{equation}
		where the first inequality follows from Karamata's inequality.
\end{proof}
\begin{proof}[Proof of Lemma \ref{p:crkl}.]
	If the sequence $\{x_k\}$ is finite, with $x_m =\tilde{x}$ stationary for some $m \geq 0$, we define $x_k = x_{m}$ for every $k \geq m$, so that we can always assume $\{x_k\}$ infinite. Notice that with this convention the sufficient decrease condition \eqref{dq} is still satisfied for every $k$. Let $f_k = f(x_k) - f(x^*)$. $\{f_k\}$ is monotone decreasing by \eqref{dq}, and nonnegative since \eqref{hp:klP} holds for every $x_k$. \\ We want prove 
	$f_{k + 1} \leq q f_k$. This is clear if $f_{k + 1} = 0$. Otherwise using the notation of Proposition \ref{keyprop}  we have
	\begin{equation} \label{lindecr}
		f_k - f_{k+1} \geq \frac{L}{2}\n{x_k - x_{k+1}}^2 \geq \frac{LK^2}{2}\n{\pi(T_{\OM}(\tilde{x}_k), -\nabla f(\tilde{x}_k))} \, ,	
	\end{equation}
	where we used \eqref{dq} in the first inequality, \eqref{dq2} in the second. 
	Since $\tilde{x}_k \in \{y_j\}_{j = 0}^T$ by Proposition \ref{keyprop}, we can apply \eqref{hp:klP} in $\tilde{x}_k$ to obtain
	\begin{equation} \label{linedcr2}
		\frac{LK^2}{2}\n{\pi(T_{\OM}(\tilde{x}_k), -\nabla f(\tilde{x}_k))}^2	\geq \mu L K^2(f(\tilde{x}_k) - f(x^*)) \geq \mu L K^2f_{k+1}.
	\end{equation}
	Concatenating \eqref{lindecr}, \eqref{linedcr2} and rearranging we obtain
	\begin{equation} \label{eq:decrease}
		f_{k+1} \leq (1 + \mu LK^2)^{-1}f_k = q f_k \, .
	\end{equation}
	Thus by induction for any $i \geq 0$
	\begin{equation} \label{eq:qdec}
		f_{k + i} \leq q^i f_k \, ,
	\end{equation}
	which implies in particular \eqref{th:qdec}. \\ 
	We can now bound the length of the tails of $\{x_k\}$:
	\begin{equation} \label{eq:tailL}
		\sum_{i = 0}^{+\infty} \n{x_{k + i} - x_{k + i + 1}} \leq \sqrt{\frac{2}{L}} \sum_{i = 0}^{+ \infty} \sqrt{f_{k + i} - f_{k + i + 1}} \leq  \frac{\sqrt{2f_k(1 - q)}}{\sqrt{L}(1 - \sqrt{q})} \leq \frac{\sqrt{2f_0(1 - q)}}{\sqrt{L}(1 - \sqrt{q})} q^{\frac{k}{2}} \, ,
	\end{equation}
	where we used \eqref{dq} in the first inequality,  Lemma \ref{l:karamata} with $\{\tf_i\} = \{f_{k + i}\}$ and for $j \rightarrow +\infty$ in the second inequality, and \eqref{eq:qdec} in the third.
	In particular $x_k \rightarrow \tilde{x}^*$ with 
	\begin{equation}
		\n{x_k - \tilde{x}^*} \leq \sum_{j = 0}^{+\infty} \n{x_{k + j} - x_{k + j + 1}} = \frac{\sqrt{2f_0(1 - q)}}{\sqrt{L}(1 - \sqrt{q})} q^{\frac{k}{2}} 
	\end{equation}
	by \eqref{eq:tailL}. 
\end{proof}
\begin{proof}[Proof of Theorem \ref{t:loc}]
By continuity, for $\tilde{\delta} \rightarrow 0$ and $f_0 = f(x_0) - f(x^*)$ we have that 
\begin{equation}
    \max_{x_0 \in B_{\tdelta}(x^*) \cap [f \geq f(x^*)]} f_0 \rightarrow 0 \, ,
\end{equation}
so we can take $\tilde{\delta} < \delta/2$ small enough in such a way that
\begin{equation}\label{eq:tdeltamax}
    \max_{x_0 \in B_{\tdelta}(x^*) \cap [f \geq f(x^*)]}  \frac{\sqrt{2f_0(1 - q)}}{L(1 - \sqrt{q})} + \sqrt{\frac{2}{L}}\sqrt{f_0} < \frac{\delta}{2} \, .
\end{equation}
	Let now $x_0 \in B_{\tdelta}(x^*) \cap [f \geq f(x^*)]$, so that
	\begin{equation}\label{eq:tdelta}
		\tdelta < \frac{\delta}{2} < \delta - \frac{\sqrt{2f_0(1 - q)}}{L(1 - \sqrt{q})} - \sqrt{\frac{2}{L}}\sqrt{f_0} \, ,
	\end{equation}	
where we use \eqref{eq:tdeltamax} in the second inequality. 
		We now want to prove, by induction on $k$, $\{x_i\}_{i \in [0:k]} \subset B_{\delta}(x^*)$ with $f(x_{i + 1}) \leq qf(x_i)$ for every $i \in [0:k]$ and $k \in \mathbb{N}$. To start with, 
		\begin{equation} \label{eq:firsttail}
			\sum_{i = 0}^{k - 1} \n{x_i - x_{i + 1}} \leq \sqrt{\frac{2}{L}} \sum_{i = 0}^{k - 1} \sqrt{f_{i} - f_{i + 1}} \leq \frac{\sqrt{2f_0(1 - q)}}{\sqrt{L}(1 - \sqrt{q})}
		\end{equation}
		where we used \eqref{dq} in the first inequality, and Lemma \ref{l:karamata} (which we can apply thanks to the inductive assumption) in the second. But then
		\begin{equation}\label{eq:inball}
			\begin{aligned}
				& \n{x_{k + 1} - x^*} \leq \n{x_0 - x^*} + \left( \sum_{i = 0}^{k - 1} \n{x_i - x_{i + 1}} \right) + \n{x_k - x_{k + 1}} \\ 
				& \leq \tdelta + \frac{\sqrt{2f_0(1 - q)}}{L(1 - \sqrt{q})} + \sqrt{\frac{2}{L}}\sqrt{f_k - f_{k + 1}} \\
				& <	\tdelta + \frac{\sqrt{2f_0(1 - q)}}{L(1 - \sqrt{q})} + \sqrt{\frac{2}{L}}\sqrt{f_k} < \delta \, ,	
			\end{aligned}
		\end{equation}
		where we used \eqref{eq:firsttail} together with \eqref{dq} in the second inequality, the assumption $x_k \in B_{\delta}(x^*) \Rightarrow f_{k + 1} \geq 0$ in the third inequality, and \eqref{eq:tdelta} together with $f_0 \geq f_{k}$ in the last inequality. \\
		
		We now have
				\begin{equation}\label{eq:inball}
			\begin{aligned}
				& \n{\tilde{x}_{k} - x^*} \leq \n{x_0 - x^*} + \left( \sum_{i = 0}^{k - 1} \n{x_i - x_{i + 1}} \right) + \n{x_k - \tilde{x}_{k}} \\ 
				& \leq \n{x_0 - x^*} + \left( \sum_{i = 0}^{k - 1} \n{x_i - x_{i + 1}} \right) + \n{x_k - x_{k + 1}} < \delta \, ,	
			\end{aligned}
		\end{equation}
		where we use $\n{\tilde{x}_k - x_k} \leq \n{x_{k + 1} - x_k}$ in the second inequality and the last inequality follows as in \eqref{eq:inball}. Thus $\tilde{x}_k \in B_{\delta}(x^*)$ as well, which is enough to prove \eqref{eq:decrease} and complete the induction. We have thus obtained $\{\tilde{x}_k\}, \{ x_k \} \subset B_{\delta}(x^*)$, and the conclusion follows exactly as in the proof of Lemma \ref{p:crkl}.
\end{proof}

\begin{proof}[Proof of Corollary \ref{cor:glob}]
    Let $x^*$ be a limit point of $\{x_k\}$, and let $\tilde{\delta}$ be as in Theorem~\ref{t:loc}. First, for some $\bar{k} \in \mathbb{N}$ we must have $x_{\bar{k}} \in B_{\tilde{\delta}}(x^*)$. Furthermore, for every $k \in \mathbb{N}$ we have $f(x_k) \geq f(x^*)$ because $f(x_k)$ is non increasing and converges to $f(x^*)$. Thus we have all the necessary assumptions to obtain the asymptotic rates by applying Theorem \ref{t:loc} to $\{y_k\}=\{x_{\bar{k} + k}\}$.
\end{proof}
	\begin{lemma} \label{l:maxstepsize}
		Let $x$ be a proper convex combination of atoms in $A' \subset A$, and $d \neq 0$ feasible direction in $x$. Then, for some $y \in \conv(A')$, we have 
		\begin{equation}
			\hat{\alpha}^{\max}(y, d) \geq \frac{\PWidth(A)}{\n{d}} \, .
		\end{equation}
	\end{lemma}
	\begin{proof}
		Let $y\in \argmax_{z \in \conv(A')} \hat{\alpha}^{\max}(z, d)$, and let $A'' \subset A'$ be such that $y$ is a proper convex combination of elements in $A''$. Furthermore, let $\F_y$ be the minimal face containing the maximal feasible step point $\bar{y} := y + \hat{\alpha}^{\max}(y, d)$. We claim that $\F_y \cap A'' = \emptyset$. In fact, for $p \in A'' \cap \F_y$ we can consider an homothety of center $p$ and factor $1 + \epsilon$ mapping $y$ in $y_{\epsilon} \in \conv(A'')$ and $\bar{y}$ in $\bar{y}_{\epsilon} \in \F_y$ with 
		$$\bar{y}_{\epsilon} = y_{\epsilon} + (1 + \epsilon) \hat{\alpha}^{\max}(y, d) d \, .$$
		But then we would have $\hat{\alpha}(\bar{y}_{\epsilon}, d) \geq (1 + \epsilon) \hat{\alpha}(\bar{y}, d)$, in contradiction with the maximality of $\hat{\alpha}(\bar{y}, d)$. Therefore
		\begin{equation}
			\hat{\alpha}^{\max}(y, d) \geq \dist(A'', \F_y) \geq \min_{\F \in \tx{pfaces}(\OM)} \dist(\F, \conv(A\setminus \F)) = \PWidth(A) \, ,
		\end{equation}
	where we used $A'' \cap \F = \emptyset$ in the second inequality, and \cite[Theorem 2]{pena2018polytope} in the equality. 
	\end{proof}

\textbf{Data availability.} The data analysed during the current study are available in the 2nd DIMACS implementation challenge repository, 
\begin{verbatim}
  http://archive.dimacs.rutgers.edu/pub/challenge/graph/benchmarks/clique/  
\end{verbatim}

 \textbf{Conflict of interest.} The authors have no competing interests to declare that are relevant to the content of this article. 
\bibliographystyle{plain}
\bibliography{tesibib}

\end{document}